\documentclass[a4paper,reqno,11pt, oneside]{amsart}
\usepackage[pdftex]{graphicx,xcolor,hyperref}
\usepackage{amssymb}
\usepackage{amstext}
\usepackage{amsmath}
\usepackage{amscd}
\usepackage{amsthm}
\usepackage{amsfonts}
\usepackage{enumerate}
\usepackage{caption}
\usepackage{color}
\usepackage{here}
\usepackage{bm}
\usepackage{tikz}
\usepackage{calc}
\usepackage{multirow}
\usepackage{makecell}
\usepackage[whole]{bxcjkjatype}
\usepackage[colorinlistoftodos]{todonotes}
\usetikzlibrary{decorations.markings}
\usetikzlibrary{positioning}
\usetikzlibrary{arrows}
\usetikzlibrary{patterns,decorations.pathreplacing}
\usetikzlibrary{calc}
\tikzstyle{vertex}=[circle ,draw, inner sep=0pt, minimum size=6pt]


\makeatletter
  
  \@addtoreset{equation}{section}
\makeatother

\newcommand{\Ac}{\mathcal{A}}
\newcommand{\Bc}{\mathcal{B}}
\newcommand{\Cc}{\mathcal{C}}
\newcommand{\Dc}{\mathcal{D}}
\newcommand{\Ec}{\mathcal{E}}
\newcommand{\Fc}{\mathcal{F}}

\newcommand{\Oc}{\mathcal{O}}

\newcommand{\Sc}{\mathcal{S}}
\newcommand{\Tc}{\mathcal{T}}

\newcommand{\ZZ}{\mathbb{Z}}

\newcommand{\RR}{\mathbb{R}}

\newcommand{\TT}{\mathbb{T}}
\newcommand{\kk}{\Bbbk}

\newcommand{\Bf}{\mathfrak{B}}
\newcommand{\If}{\mathfrak{I}}

\newcommand{\xb}{\mathbf{x}}
\newcommand{\yb}{\mathbf{y}}

\newcommand{\set}[1]{\left\{ #1 \right\}}
\newcommand{\rbra}[1]{\left( #1 \right)}

\definecolor{suprise}{rgb}{1, 0.0, 0.69}

\definecolor{kojigreen}{rgb}{0, 1, 0}

\def\opn#1#2{\def#1{\operatorname{#2}}} % to make operators
\opn\conv{conv} \opn\dep{depth} \opn\Spec{Spec} \opn\cone{cone} \opn\ini{in} \opn\codeg{codeg} \opn\deg{deg}
\opn\Graph{Graph} \opn\sign{sign} \opn\Ehr{Ehr} \opn\rank{rank} \opn\type{type} \opn\reg{reg} \opn\cl{cl}
\opn\Hilb{Hilb} \opn\Indeg{Indeg} \opn\link{link} \opn\Tor{Tor} \opn\MNF{MNF} \opn\dep{depth} \opn\Stab{Stab}

\newtheorem{thm}{Theorem}[section]
\newtheorem{cor}[thm]{Corollary}
\newtheorem{lem}[thm]{Lemma}
\newtheorem{prop}[thm]{Proposition}

\newtheorem{q}[thm]{Question}

\theoremstyle{definition}

\theoremstyle{remark}
\newtheorem{rem}[thm]{Remark}

\title{Matroid polytopes with small rank}
\author{Masato Konoike, Koji Matsushita}

\address[M. Konoike]{Department of Pure and Applied Mathematics, Graduate School of Information Science and Technology, Osaka University, Suita, Osaka 565-0871, Japan}
\email{kounoike-m@ist.osaka-u.ac.jp}

\address[K. Matsushita]{Department of Pure and Applied Mathematics, Graduate School of Information Science and Technology, Osaka University, Suita, Osaka 565-0871, Japan}
\email{k-matsushita@ist.osaka-u.ac.jp}


\begin{document}

\maketitle

\begin{abstract}
For a lattice polytope $P$, the rank of $P$ is defined by $F-(\dim P+1)$, where $F$ is the number of facets of $P$.
In this paper, we study matroid polytopes with small rank. More precisely, we characterize matroid independence polytopes and graphic matroid base polytopes with rank at most three. Furthermore, using this characterization, we investigate their relationships with order polytopes, stable set polytopes, and edge polytopes.
\end{abstract}

\section{Introduction}
Matroid independence polytopes (resp. matroid base polytopes) are defined as the convex hull of the indicator vectors of all independent sets (resp. bases) of matroids.
These polytopes have been studied from various perspectives, including combinatorics, commutative algebra, and algebraic geometry.
For example, normality (\cite{Herzog2002Discrete,white1977basis}), defining ideals of toric rings (\cite{Blasiak2008TheToric,Lason2014OnThe,white1980AUnique}), Gorensteinness (\cite{hibi2021gorenstein,kolbl2020gorenstein,Lason2023GorMat}) and nearly Gorensteinness (\cite{hall2023nearly}) of matroid polytopes have been investigated.
Moreover, the relationships between them and other classes, such as order polytopes and stable set polytopes, are also studied (\cite{aliniaeifard2018stable,benedetti2023lattice}).
In this paper, we focus on matroid polytopes with small rank and analyze their properties and their relationships with other polytopes.

For a lattice polytope $P$ with $\dim P\ge 1$, the \textit{rank} of $P$ is defined by
\[
\rank P:=F-(\dim P+1),
\]
where $F$ denotes the number of facets of $P$.
Note that $\rank P$ is a nonnegative integer.
The rank of $P$ is a notable invariant due to its commutative ring-theoretic meaning; $\rank P$ is equal to the rank of the divisor class group of the toric ring $\kk[P]$ of $P$, where $\kk$ is a field, if $\kk[P]$ is normal (cf. \cite[Theorem~9.8.19]{villarreal2001monomial}).
In \cite{matsushita2022three,matsushita20240/1}, $(0,1)$-polytopes have been investigated via their rank.
In particular, in the paper \cite{matsushita2022three}, ordered polytopes, stable set polytopes, and edge polytopes with small rank have been classified up to unimodular equivalence by characterizing the underlying combinatorics objects.
Here, we say that two lattice polytopes $P,P' \subset \RR^d$ are \textit{unimodularly equivalent} 
if there are a lattice vector ${\bf v} \in \ZZ^d$ and a unimodular transformation $f \in \mathrm{GL}_d(\ZZ)$ such that $P'=f(P)+{\bf v}$.
We write $P\cong P'$ if two lattice polytopes $P$ and $P'$ are unimodularly equivalent.
It is known that if two lattice polytopes $P,P' \subset \RR^d$ have the integer decomposition property, then one has $P\cong P'$ if and only if $\kk[P]$ and $\kk[P']$ are isomorphic as $\kk$-algebras (cf. \cite[Theorem~5.22]{bruns2009polytopes}).

The goal of this paper is to build on these works by classifying matroid polytopes with small rank (in other words, classifying the toric rings of matroid polytopes whose divisor class groups have small rank), and comparing them with other classes of polytopes.

For a graph $G$, let $M(G)$ be the graphic matroid.
We characterize matroid independence polytopes and graphic matroid base polytopes with rank at most three as follows:

\begin{thm}[{Theorem~\ref{thm: classifyIM} and Corollary~\ref{cor:independence}}]\label{mainthm1}
    Let $M$ be a connected matroid and let $P(M)$ be the matroid independence polytope of $M$. Then we have $\rank P(M)=0$ or $\rank P(M)\ge 3$. Moreover, we can see that
    \begin{enumerate}[\normalfont(1)]
        \item $\rank P(M)=0$ if and only if $P(M)\cong P(M(\Ac_s))$ for some $s\in \ZZ_{>0}$.
        %\item There is no matroid independence polytope $P(M)$ with rank 1 or rank 2.
        \item $\rank P(M)=3$ if and only if $P(M)\cong P(M(\Dc_{s_1,s_2,s_3}))$ for some $s_1, s_2, s_3\in \ZZ_{\ge 0}$.
    \end{enumerate}
    In particular, if $\rank P(M)\le 3$, then $P(M)$ is a graphic matroid independence polytope.
\end{thm}

\begin{thm}[{Theorem~\ref{thm:smallmultibasepolytope}}]\label{mainthm2}
    Let $M$ be the graphic matroid of a 2-connected (not necessarily simple) graph, and let $B(M)$ be the matroid base polytope of $M$.
    Then we have the following:
    \begin{enumerate}[\normalfont(1)]
        \item $\rank B(M)=0$ if and only if $B(M)\cong B(M(\Ac_{s+1}))$ for some $s\in \ZZ_{>0}$.
        \item $\rank B(M)=1$ if and only if $B(M)\cong B(M(\Bc_{s,p}))$ for some $s,p\in \ZZ_{>0}$.
        \item $\rank B(M)=2$ if and only if $B(M)\cong B(M(\Bc_{s_1,s_2,p}))$ for some $s_1,s_2\in \ZZ_{>0}$ and $p\in \ZZ_{\ge 0}$.
        \item $\rank B(M)=3$ if and only if $B(M)$ is unimodularly equivalent to one of the following matroid base polytopes:
        \begin{itemize}
        \item[(i)] $B(M(\Bc_{s_1,s_2,s_3,p}))$ for some $s_1,s_2,s_3\in \ZZ_{>0}$ and $p\in \ZZ_{\ge 0}$;
        \item[(ii)] $B(M(\Cc_{s,t,p,q}))$ for some $s,q\in \ZZ_{>0}$ and $t,p\in \ZZ_{\ge 0}$;
        \item[(iii)] $B(M(\Dc_{s_1+1,s_2+1,s_3+1}))$ for some $s_1,s_2,s_3\in \ZZ_{> 0}$.
        \end{itemize}
    \end{enumerate}
\end{thm}
\noindent See Section~\ref{subsec:graph} for the definitions of the graphs $\Ac_s$, $\Bc_{s_1,\ldots,s_n,p}$, $\Cc_{s,t,p,q}$ and $\Dc_{s_1,s_2,s_3}$.

Moreover, we examine the relationships among the following five equivalent classes:

\noindent${\bf MI}_n:=\set{\text{matroid independence polytopes with rank $n$}}/\cong$;

\noindent${\bf GMB}_n:=\set{\text{graphic matroid base polytopes with rank $n$}}/\cong$;

\noindent${\bf Order}_n:=\set{\text{order polytopes with rank $n$}}/\cong$;

\noindent${\bf Stab}_n:=\set{\text{stable set polytopes of perfect graphs with rank $n$}}/\cong$;

\noindent${\bf Edge}_n:=\set{\text{edge polytopes of graphs satisfying the odd cycle condition with rank $n$}}/\cong$.

\begin{thm}[{Propositions~\ref{prop:n=01},\ref{prop:n=2} and \ref{prop:n=3}}]\label{mainthm3}
We have the following:
\begin{enumerate}[\normalfont(1)]
    \item ${\bf MI}_0={\bf GMB}_0={\bf Stab}_0$.
    \item ${\bf MI}_n\subsetneq {\bf GMB}_n\subsetneq {\bf Stab}_n$ for $n=1,2$.
    \item For any $\Sc \in \set{{\bf MI}_2,{\bf GMB}_2}$, there is no inclusion between $\Sc$ and ${\bf Edge}_2$.
    \item For any two distinct elements $\Sc,\Tc\in \{{\bf MI}_3, {\bf GMB}_3,{\bf Order}_3, {\bf Stab}_3, {\bf Edge}_3\}$, there is no inclusion between $\Sc$ and $\Tc$.
    In addition, we have ${\bf MI}_3 \not\subset {\bf Order}_3\cup{\bf Stab}_3\cup{\bf Edge}_3\cup{\bf GMB}_3$ and ${\bf GMB}_3\subsetneq {\bf Stab}_3\cup{\bf Edge}_3$.
\end{enumerate}
\end{thm}

\medskip

The structure of this paper is as follows.
In Section~\ref{sec:pre}, we prepare the materials required for later discussion.
We recall definitions and notation associated with (graphic) matroids and graphs.
In Section~\ref{sec:main}, we characterize matroid independence polytopes and matroid base polytopes with rank at most 3, respectively.
In Section~\ref{sec:relation}, we study the relationships among ${\bf MI}_n$, ${\bf GMB}_n$, ${\bf Order}_n$, ${\bf Stab}_n$ and ${\bf Edge}_n$ in the case where $n \leq 3$.

\subsection*{Acknowledgements}
The second author is partially supported by Grant-in-Aid for JSPS Fellows Grant JP22J20033.

\section{Preliminaries}\label{sec:pre}

The goal of this section is to prepare the required objects for discussions of our main results.

\subsection{Matroids and matroid polytopes}
First, we recall matroids and their matroid polytopes (consult, e.g., \cite{oxley2011matroid} for the introduction).
Throughout this subsection, let $M$ be a matroid on the ground set $E$ with the set of independent sets $\If$, the set of bases $\Bf$ and the rank function $r:2^E \to \ZZ_{\geq 0}$.

We introduce some fundamental objects associated with matroids.
\begin{itemize}
    \item For a subset $A\subset E$, let $\cl(A)$ be the closure of $A$, that is, $\cl(A):=\{e\in E : r(A\cup\set{e})=r(A)\}$. We say that a subset $F \subset E$ is a \textit{flat} if $\cl(F) = F$.
    \item A subset $A\subset E$ is called \textit{indecomposable} if $A$ cannot be decomposed into proper subsets $A = A_1 \sqcup A_2$ such that $r(A) = r(A_1) + r(A_2)$. 
    \item A set $A$ is called \textit{connected} if every two elements of A belong to some circuit contained in $A$.
    \item We say that $M$ is \textit{connected} if $E$ is connected.
    We can see that $M$ is connected if and only if $E$ is indecomposable if and only if $M$ is not a direct sum of two or more nontrivial matroids.
    \item We say that $A\subset E$ is a \textit{flacet} if $A$ is a flat such that the restriction of $M$ to $A$ and the contraction of $A$ in $M$ are connected.
\end{itemize}

\medskip

For a set $A$ and its subset $S$, let $\chi_S$ be the characteristic vector of $S$ in $\RR^A$, that is, $\chi_S(a)=1$ if $a\in S$ and $\chi_S(a)=0$ otherwise.

We define two lattice polytopes associated with $M$:
\[
P(M):=\conv(\set{\chi_I : I\in \If}) \quad  \text{ and } \quad B(M):=\conv(\set{\chi_B : B\in \Bf}).
\]
We call $P(M)$ (resp. $B(M)$) the \textit{(matroid) independence polytope} (resp. \textit{(matroid) base polytope}) of $M$.

We also define the \textit{product} of two polytopes $P\subset \RR^d$ and $Q\subset \RR^e$ as
$$P\times Q=\{(\xb,\yb) : \xb\in P,\yb\in Q\}\subset \RR^{d+e}.$$
We can see that $P\times Q$ is a polytope of dimension $\dim P + \dim Q$, whose nonempty faces are the products of nonempty faces (including itself) of $P$ and $Q$.
In particular, the number of facets of $P\times Q$ is equal to $F_P+F_Q$, where $F_P$ and $F_Q$ denote the numbers of facets of $P$ and $Q$, respectively.
Therefore, we have $\rank P\times Q= \rank P + \rank Q +1$.

If $M$ has two connected components $M_1$ and $M_2$, then we have $P(M)=P(M_1)\times P(M_2)$ and $B(M)=B(M_1)\times B(M_2)$.
In particular, if $M$ has a loop $e$, then $P(M)$ (resp. $B(M)$) is unimodularly equivalent to $P(M\setminus e)$ (resp. $B(M\setminus e)$).
Thus, in classifying matroid polytopes with small rank, it suffices to consider connected matroids.

Let $\Bf^\ast:=\{E - B : B \in \Bf\}$. Then $\Bf^\ast$ forms the set of bases of a matroid on $E$, called the \textit{dual} matroid $M^\ast$ of $M$.
We can see that $B(M^\ast) = (1, \ldots,1) - B(M)$, and hence $B(M)$ and $B(M^\ast)$ are unimodularly equivalent.

\medskip

The dimension and facet descriptions of the independence polytope and base polytope of a matroid are already given as follows:

\begin{lem}[{\cite[Theorem 40.5]{schrijver2003combinatorial}}]\label{facets:independent}
    Suppose that $M$ has no loops. 
    Then we have $\dim P(M)=|E|$ and
    \begin{align*}
    P(M) 
    =
    \left\{ 
    x \in \mathbb{R}^{E} \; : \;
    \begin{array}{rll}
        x(e) &\geq 0, & \text{ for every } e \in E \vspace{0.2cm}\\
        \sum_{e \in F}x(e) & \leq r(F), &\text{ for every indecomposable flat }F 
    \end{array}
    \right\}.
\end{align*}
    In particular, each supporting hyperplane of the above half spaces defines a facet of $P(M)$.
\end{lem}

\begin{lem}[{\cite{feichtner2004matroid,kim2010flag}}]\label{facets:base}
    Suppose that $M$ is connected.
    Then we have $\dim B(M)=|E|-1$ and 
    \begin{align*}
    B(M) 
    =
    \left\{ 
    x \in \mathbb{R}^{E} \; : \;
    \begin{array}{rll}
        x(e) &\geq 0, & \text{ for every $e \in E$ such that } \\
        \; & \; & \text{ $M \setminus e$ is connected,}\vspace{0.2cm} \\
        \sum_{e \in F}x(e) & \leq r(F), &\text{ for every flacet }\emptyset \neq F \subsetneq E
    \end{array}
    \right\}.
   \end{align*}
    In particular, each supporting hyperplane of the above half spaces defines a facet of $B(M)$.
\end{lem}

\subsection{Graphs and graphic matroids}\label{subsec:graph}
Next, we recall graphs and their graphic matroids.
Throughout this subsection, let $G$ be a finite undirected (not necessarily simple) graph on the vertex set $V(G)$ with the edge set $E(G)$.

We use the following notation on graph theory:
\begin{itemize}
    \item For a subset $S \subset V(G)$, let $G[S]$ denote the induced subgraph with respect to $S$ and let $G\setminus S:=G[V(G)\setminus S]$ (for a vertex $v\in V(G)$, we denote by $G \setminus v$ instead of $G \setminus \{v\}$).
    \item For a subset $S \subset E(G)$, let $G\setminus S$ (resp. $G/S$) denote the graph obtained from $G$ by removing (resp. contracting) the edges in $S$ (for an edge $e\in E(G)$, we denote by $G \setminus e$ (resp. $G/e$) instead of $G \setminus \{e\}$ (resp. $G/\set{e}$)).
\end{itemize}

We can obtain a matroid from $G$ called the \textit{graphic matroid} (denoted by $M(G)$), which is a matroid on the set $E(G)$ whose independent sets are the forests in $G$. A set $B \subset E(G)$ is a basis of $M(G)$ if and only if $B$ is a spanning forest of $G$.
We can see that the graphic matroid $M(G)$ is connected if and only if $G$ is 2-connected, that is, $G\setminus v$ is connected for any vertex $v\in V(G)$ (in this paper, we consider the complete graph $K_2$ on 2 vertices as a 2-connected graph).

It is well known that $G$ is 2-connected if and only if $G$ has an (open) ear decomposition (see \cite[Proposition~3.1.1]{diestel2017graph}), that is, $G$ can be decomposed as $K_2 \cup P_1 \cup \cdots \cup P_r$, where $P_i$ is a path, and $(K_2 \cup P_1 \cup \cdots \cup P_{i-1}) \cap P_i$ consists of the two end vertices of $P_i$ for each $i$.

If $G$ is a plane graph, we can form another plane graph called the \textit{dual graph} $G^\ast$. The vertices of $G^\ast$ correspond to the faces of $G$, with each vertex being placed in the corresponding face. Every edge $e$ of $G$ gives rise to an edge of $G^\ast$ joining the two faces of $G$ that contain $e$.
If $G$ is a planar graph, then $M(G^\ast) \cong M^\ast(G)$ holds, which implies that $\rank B(M(G^\ast)) = \rank B(M(G))$.

In what follows, we assume that $G$ is 2-connected.
An edge $e\in E(G)$ is said to be \textit{contractible} (resp. \textit{deletable}) if $G/e$ (resp. $G\setminus e$) is also 2-connected. 
It is known that $G\setminus e$ or $G/e$ is 2-connected, that is, $e$ is deletable or contractible (\cite[Theorem~4.3.1]{oxley2011matroid}).
Moreover, we can see that $F\subset E(G)$ is a flacet of $M(G)$ if and only if there exists a 2-connected induced subgraph $H$ of $G$ with $E(H)=F$ such that $G/F$ is 2-connected.
We call such a subgraph $H$ a \textit{flacet} of $G$.
In particular, if an edge $e\in E(G)$ is contractible, then $\set{e}$ is a flacet of $G$.
Therefore, it follows from Lemma~\ref{facets:base} that the deletable edge $e\in E(G)$ (resp. the flacet $H$ of $G$) one-to-one corresponds to the facet defining inequality $x(e)\ge 0$ (resp. $\sum_{e\in E(H)}x(e)\le r(E(H))$).

\bigskip

In the remainder of this section, we define some graphs needed to state our main results.

The \emph{complete multipartite graph} $K_{r_1, ..., r_n}$ ($n > 1$) is the simple graph on the vertex set of the form $\bigsqcup_{i=1}^n V_i$ ($|V_i| = r_i$ for $i = 1, ..., n$) with the edge set
    \begin{equation*}
        E\left(K_{r_1, ..., r_n}\right) := \left\{\{v,v'\} : v \in V_i, v' \in V_j, 1 \leq i < j \leq n\right\}.
    \end{equation*}
When $r_1=\cdots=r_n=1$, we denote $K_{\underbrace{1,\ldots,1}_n}$ by $K_n$, which is nothing but the complete graph on $n$ vertices.

For integers $a$ and $b$ with $a\le b$, let $[a,b]:=\set{c\in \ZZ : a\le c \le b}$.
We define four graphs as follows:
\begin{itemize}
    \item For $s\in \ZZ_{>0}$, let $\Ac_s$ be the graph on $V(\Ac_s):=\set{v_1,v_2}$ with 
    \[
    E(\Ac_s):=\set{e_i:=\set{v_1,v_2} :  i\in [1,s]}.
    \]
    %where $\times$ stands for the multiplicities. 
    See Figure~\ref{graph_A}.
    \item For $n\in \ZZ_{>0}$, $s_1,\ldots,s_n\in\ZZ_{>0}$ and $p\in \ZZ_{\ge 0}$, let $\Bc_{s_1,\ldots,s_n,p}$ be the graph on 
    \[
    V(\Bc_{s_1,\ldots,s_n,p}):=\set{v_1,\ldots,v_{n+1}}\cup\set{u_1,\ldots,u_p}
    \]
    with 
    \begin{align*}
    E(\Bc_{s_1,\ldots,s_n,p}):=&\rbra{\bigcup_{i\in [1,n]}\set{e_{i,j}:=\set{v_i,v_{i+1}} : j\in [1,s_i+1]}}\cup \\ &\set{\epsilon_i:=\set{u_i,u_{i+1}} : i\in [0,p]},
    %\set{\set{v_1,u_1},\set{u_1,u_2},\ldots,\set{u_{p-1},u_p},\set{u_p,v_2}}.
    \end{align*}
    where we let $u_0=v_1$ and $u_{p+1}=v_{n+1}$. See Figure~\ref{graph_B}.
    \item For $s,q\in \ZZ_{>0}$ and $t,p\in \ZZ_{\ge 0}$, let $\Cc_{s,t,p,q}$ be the graph on 
     \[
     V(\Cc_{s,t,p,q}):=\set{v_1,v_2,v_3}\cup\set{u_1,\ldots,u_p}\cup\set{w_1,\ldots,w_q}
     \]
     with 
    \begin{align*}
    E(\Cc_{s,t,p,q}):=&\set{e_{1,i}:=\set{v_1,v_2} : i\in[1,s+1]}\cup\set{e_{2,i}:=\set{v_2,v_3} : i\in[1,t+1]}\cup \\ &\set{\epsilon_{1,i}:=\set{u_i,u_{i+1}} : i\in [0,p]}\cup\set{\epsilon_{2,i}:=\set{w_i,w_{i+1}} : i\in [0,q]},
    \end{align*}
    where we let $u_0=v_1$, $w_0=v_2$ and $u_{p+1}=w_{q+1}=v_3$.
    See Figure~\ref{graph_C}.
    \item For $s_1,s_2,s_3\in \ZZ_{\ge 0}$, let $\Dc_{s_1,s_2,s_3}$ be the graph on 
     $V(\Dc_{s_1,s_2,s_3}):=\set{v_1,v_2,v_3}$ with 
    \begin{align*}
    E(\Dc_{s_1,s_2,s_3}):=\bigcup_{i=1}^3\set{e_{i,j}:=\set{v_i,v_{i+1}} : j\in[1,s_i]},
    \end{align*}
    where we let $v_4=v_1$. See Figure~\ref{graph_D}.
\end{itemize}

%%%%%%%%%%%%%%%%%%%%%%%
%%%%%%%%%%%%%%%%%%%%%%%%%%%%%%%%%%%%%%%%%%%%%%%%%%%%%%%%%%%%%%%%%%%%%%
\begin{figure}[H]
    \centering
    \begin{minipage}{0.46\linewidth}
        \centering
        \scalebox{0.8}{
            \begin{tikzpicture}
                \coordinate (N1) at (0,2); 
                \coordinate (N2) at (4,2);
                %edge
                %\draw  (N1) .. controls (3,3.5) .. (N2); 
                \draw  (N1) to [out=30,in=150] (N2);
                \draw  (N1) to [out=-30,in=-150] (N2);
                \draw[dotted]  (2,2.3)--(2,1.7);
                %node
                \draw [line width=0.05cm, fill=white] (N1) circle [radius=0.15] node[above left] {\Large $v_1$};
                \draw [line width=0.05cm, fill=white] (N2) circle [radius=0.15] node[above right] {\Large $v_2$};
                \node[] at (2,1) {\Large $e_s$};
                \node[] at (2,3) {\Large $e_1$};
                \node[] at (2,0) {$\;$};
                \node[] at (2,4.5) {$\;$};
            \end{tikzpicture}}
        \caption{The graph $\Ac_s$}
        \label{graph_A}
    \end{minipage}
    \begin{minipage}{0.53\linewidth}
        \centering
        \scalebox{0.8}{
        \begin{tikzpicture}
                \coordinate (N1) at (0,2); 
\coordinate (N2) at (2,2); 
\coordinate (N3) at (1,0); 
\coordinate (N4) at (5,0); 
\coordinate (N5) at (6,2);
\coordinate (N6) at (4,2);
%edge
%\draw  (N1) .. controls (1,3.5) .. (N2); 
\draw  (N1) to [out=60,in=120] (N2);
\draw  (N1) to [out=-60,in=-120] (N2);
%\draw  (N2) .. controls (3,3.5) .. (N6); 
%\draw  (N2) to [out=60,in=120] (N6);
%\draw  (N2) to [out=-60,in=-120] (N6);
%\draw  (N6) .. controls (5,3.5) .. (N5); 
\draw  (N6) to [out=60,in=120] (N5);
\draw  (N6) to [out=-60,in=-120] (N5);
\draw  (N1)--(N3);
\draw  (N4)--(N5);
\draw  (N3)--(2,0);
\draw  (N4)--(4,0);
\draw[dotted]  (2.5,0)--(3.5,0);
\draw[dotted]  (1,2.3)--(1,1.7);
\draw[dotted]  (2.5,2)--(3.5,2);
\draw[dotted]  (5,2.3)--(5,1.7);

%node
\draw [line width=0.05cm, fill=white] (N1) circle [radius=0.15] node[above left] {\Large $v_1$};
\draw [line width=0.05cm, fill=white] (N2) circle [radius=0.15] node[] at (2,2.5) {\Large $v_2$};
\draw [line width=0.05cm, fill=white] (N3) circle [radius=0.15] node[above right] {\Large $u_1$};
\draw [line width=0.05cm, fill=white] (N4) circle [radius=0.15] node[above left] {\Large $u_p$};
\draw [line width=0.05cm, fill=white] (N5) circle [radius=0.15] node[above right] {\Large $v_{n+1}$};
\draw [line width=0.05cm, fill=white] (N6) circle [radius=0.15] node[] at (4,2.5) {\Large $v_n$};
\node[] at (1.2,1.2) {\Large $e_{1,s_1+1}$};
\node[] at (1,2.8) {\Large $e_{1,1}$};
%\node[] at (3,1.2) {\Large $e_{2,t+1}$};
%\node[] at (3,2.8) {\Large $e_{2,1}$};
\node[] at (4.8,1.2) {\Large $e_{n,s_n+1}$};
\node[] at (5,2.8) {\Large $e_{n,1}$};
\node[below left] at ($(N1)!0.5!(N3)$) {\Large $\epsilon_0$};
\node[below right] at ($(N5)!0.5!(N4)$) {\Large $\epsilon_p$};
\node[] at (0,-0.5) {$\;$};
\node[] at (0,4) {$\;$};
            \end{tikzpicture}}
        \caption{The graph $\Bc_{s_1,\ldots,s_n,p}$}
        \label{graph_B}
    \end{minipage}
\end{figure}
%%%%%%%%%%%%%%%%%%%%%%%%%%%%%%%%%%%%%%%%%%%%%%%%%%%%%%%%%%%%%%%%%%%%%%

%%%%%%%%%%%%%%%%%%%%%%%%%%%%%%%%%%%%%%%%%%%%%%%%%%%%%%%%%%%%%%%%%%%%%%
\begin{figure}[H]

    \begin{minipage}{0.49\linewidth}
        \centering
        \scalebox{0.8}{
            \begin{tikzpicture}
\coordinate (N1) at (0,2); 
\coordinate (N2) at (3,2); 
\coordinate (N3) at (1,0); 
\coordinate (N4) at (5,0); 
\coordinate (N5) at (6,2);
\coordinate (N6) at (3.5,4); 
\coordinate (N7) at (5.5,4);

%edge
%\draw  (N1) .. controls (1.5,3.5) .. (N2); 
\draw  (N1) to [out=45,in=135] (N2);
\draw  (N1) to [out=-45,in=-135] (N2);
%\draw  (N2) .. controls (4.5,3.5) .. (N5); 
\draw  (N2) to [out=45,in=135] (N5);
\draw  (N2) to [out=-45,in=-135] (N5);
\draw  (N1)--(N3);
\draw  (N4)--(N5);
\draw  (N3)--(2,0);
\draw  (N4)--(4,0);
\draw[dotted]  (2.5,0)--(3.5,0);
\draw[dotted]  (1.5,2.3)--(1.5,1.7);
\draw[dotted]  (4.5,2.3)--(4.5,1.7);
\draw  (N2)--(N6);
\draw  (N7)--(N5);
\draw  (N6)--(4,4);
\draw  (5,4)--(N7);
\draw[dotted]  (4.3,4)--(4.7,4);

%node
\draw [line width=0.05cm, fill=white] (N1) circle [radius=0.15] node[above left] {\Large $v_1$};
\draw [line width=0.05cm, fill=white] (N2) circle [radius=0.15] node[] at (3,1.5) {\Large $v_2$};
\draw [line width=0.05cm, fill=white] (N3) circle [radius=0.15] node[above right] {\Large $u_1$};
\draw [line width=0.05cm, fill=white] (N4) circle [radius=0.15] node[above left] {\Large $u_p$};
\draw [line width=0.05cm, fill=white] (N5) circle [radius=0.15] node[above right] {\Large $v_3$};
\draw [line width=0.05cm, fill=white] (N6) circle [radius=0.15] node[above left] {\Large $w_1$};
\draw [line width=0.05cm, fill=white] (N7) circle [radius=0.15] node[above right] {\Large $w_q$};
\node[] at (1.5,1.1) {\Large $e_{1,s+1}$};
\node[] at (1.5,2.9) {\Large $e_{1,1}$};
\node[] at (4.5,1.1) {\Large $e_{2,t+1}$};
\node[] at (4.5,2.9) {\Large $e_{2,1}$};
\node[below left] at ($(N1)!0.5!(N3)$) {\Large $\epsilon_{1,0}$};
\node[below right] at ($(N5)!0.5!(N4)$) {\Large $\epsilon_{1,p}$};
\node[above left] at ($(N2)!0.5!(N6)$) {\Large $\epsilon_{2,0}$};
\node[above right] at ($(N5)!0.5!(N7)$) {\Large $\epsilon_{2,q}$};

            \end{tikzpicture}}
        \caption{The graph $\Cc_{s,t,p,q}$}
        \label{graph_C}
    \end{minipage}
    \begin{minipage}{0.5\linewidth}
        \centering
        \scalebox{0.8}{
            \begin{tikzpicture}
\coordinate (N1) at (0,2); 
\coordinate (N2) at (4,2); 
\coordinate (N3) at (2,5.464);
%edge
\draw  (N1) to [out=15,in=165] (N2);
\draw  (N1) to [out=-15,in=-165] (N2);
\draw  (N2) to [out=105,in=-45] (N3); 
\draw  (N2) to [out=135,in=-75] (N3);
\draw  (N1) to [out=45,in=-105] (N3); 
\draw  (N1) to [out=75,in=-135] (N3);
\draw[dotted]  (2,2.2)--(2,1.8);
\draw[dotted]  (3.166,3.85)--(2.834,3.672);
\draw[dotted]  (1.166,3.672)--(0.834,3.85);
%\draw[dotted]  (3.366,3.932)--(2.634,3.532);
%\draw[dotted]  (1.366,3.532)--(0.634,3.932);
\node[] at (2,1.4) {\Large $e_{1,s_1}$};
\node[] at (2,2.5) {\Large $e_{1,1}$};
\node[] at (2.4,3.4) {\Large $e_{2,1}$};
\node[] at (3.9,3.9) {\Large $e_{2,s_2}$};
\node[] at (0.1,4) {\Large $e_{3,s_3}$};
\node[] at (1.6,3.4) {\Large $e_{3,1}$};

%node
\draw [line width=0.05cm, fill=white] (N1) circle [radius=0.15] node[below left] {\Large $v_1$};
\draw [line width=0.05cm, fill=white] (N2) circle [radius=0.15] node[below right] {\Large $v_2$};
\draw [line width=0.05cm, fill=white] (N3) circle [radius=0.15] node[] at (2,5.964) {\Large $v_3$};
            \end{tikzpicture}}
        \caption{The graph $\Dc_{s_1,s_2,s_3}$}
        \label{graph_D}
    \end{minipage}
\end{figure}

\section{Classification of matroid polytopes with small rank}\label{sec:main}
In this section, we classify independence polytopes and base polytopes with rank at most 3, respectively.

\subsection{Independence polytopes with small rank}\label{sec:main1}
Let ${\bf MI}_n$ denote the sets of unimodular equivalence classes of matroid independence polytopes with rank $n$.
Let $M$ be a connected matroid on a ground set $E$. 
We set $\Fc_M := \{F \subset E : F \text{ is a flat of $M$ with } r(F) = 1\}$.
This is called a \textit{parallel class} and is a generalization of multiedges in graphs to the case of matroids.
The following lemma establishes some fundamental properties of $\Fc_M$:
\begin{lem}\label{lem:parallelclass}
    Let $M$ be a connected matroid (especially, $M$ has no loops). Then we have the following:
    \begin{enumerate}[\normalfont(1)]
        \item For any $F, F' \in \Fc_M$, if $F \cap F' \neq \emptyset$ then $F = F'$.
        \item $\bigsqcup_{F \in \Fc_M} F = E$.
        \item One has $|\Fc_M| \geq 1$. Moreover, the following are equivalent:
        \begin{itemize}
            \item $|\Fc_M| = 1$;
            \item $\Fc_M = \{E\}$;
            \item $M= U_{1,|E|}$ where $U_{r,n}$ denotes the uniform matroid on $n$ with rank $r$.
        \end{itemize}
    \end{enumerate}
\end{lem}
%\masato{connected matroidならlooplessなのでlooplessは必要なさそうです。}
\begin{proof}
    \begin{enumerate}[\normalfont(1)]
        \item Suppose that $F \cap F' \neq \emptyset$. 
        Note that the intersection of flats is also a flat.
        Thus, $F \cap F'$ is a flat of $M$ with $r(F \cap F') = 1$ since $F \cap F' \neq \emptyset$ and $F \cap F' \subset F$.
        If $F\neq F'$, then there exists an element $e \in (F\cup F') \setminus (F\cap F')$ and we may assume that $e\in F$. Then we have $r((F \cap F') \cup \{e\}) = 1$, which contradicts the fact that $F \cap F'$ is a flat. Therefore, we get $F = F'$.
        \item The inclusion $\bigcup_{F \in \Fc_M} F \subset E$ is trivial, so we prove that $\bigcup_{F \in \Fc_M} F \supset E$. 
        For any $x \in E$, we have $r(\cl(x))=r(x)=1$ since $M$ has no loops.
        Thus, $\cl(x)$ is a flat of $M$ with $r(\cl(x))=1$, which implies that $\cl(x)\in \Fc_M$.
        Therefore, we have $\bigcup_{F \in \Fc_M} F \supset \cl(x)\ni x$, and hence $\bigcup_{F \in \Fc_M} F \supset E$.
        From (1), we obtain $\bigsqcup_{F \in \Fc_M} F = E$.
        %$r(x) = r(\mathrm{cl}(x))$ and $\mathrm{cl}(\mathrm{cl}(x)) = \mathrm{cl}(x)$, we have $\mathrm{cl}(x) \in \mathcal{F}$. Therefore, it follows that $\bigcup_{F \in \Fc_M} F \supset E$ holds. From the above discussion and (1) of Lemma~\ref{lem:parallelclass}, we obtain $\bigsqcup_{F \in \Fc_M} F = E$.
        \item From (2), we have $\bigsqcup_{F \in \Fc_M} F = E$, which implies that $|\mathcal{F}| \geq 1$, and $|\Fc_M| = 1$ if and only if $\Fc_M = \{E\}$. If $M=U_{1,|E|}$, then $\Fc_M=\{E\}$ since $U_{1,|E|}$ has only one flat $E$ and satisfies $r(E) = 1$. Conversely, assume that $\Fc_M = \{E\}$, we have $r(E) = 1$. Since $M$ is connected, it follows that $M = U_{1,|E|}$. 
    \end{enumerate}
\end{proof}

The above lemma helps us to show the following first main theorem:

\begin{thm}\label{thm: classifyIM}
    Let $M$ be a connected matroid. Then we have the following:
    \begin{enumerate}[\normalfont(1)]
        \item $P(M) \in {\bf MI}_0$ if and only if $M$ is the uniform matroid $U_{1,|E|}$.
        \item $P(M)\notin {\bf MI}_1 \cup {\bf MI}_2$.
        \item $P(M) \in {\bf MI}_3$ if and only if $M$ is the graphic matroid $M(\Dc_{s_1, s_2, s_3})$ for some $s_1, s_2, s_3 \in \ZZ_{>0}$.
    \end{enumerate}
\end{thm}
\begin{proof}
    By Lemma~\ref{facets:independent}, we can see that $P(M) \in {\bf MI}_n$ if and only if the matroid $M$ has exactly $n+1$ indecomposable flats.
    Moreover, each element $F\in \Fc_M$ is an indecomposable flat, so we have $|\Fc_M|\le n+1$.
    Furthermore, since $E$ is also an indecomposable flat of $M$, we have $2\le |\Fc_M|\le n$ if $M\neq U_{1,|E|}$.
    \begin{enumerate}[\normalfont(1)]
        \item From Lemma~\ref{lem:parallelclass} (2) and the assumption $P(M) \in {\bf MI}_0$, we have $|\Fc_M| = 1$. Therefore, by Lemma~\ref{lem:parallelclass} (2), we have that $P(M) \in {\bf MI}_0$ if and only if $M = U_{1,|E|}$.

        \item Suppose that $P(M) \in {\bf MI}_1$. 
        Then we have $M\neq U_{1,|E|}$ and thus $2 \leq |\Fc_M| \le n= 1$, which is impossible.
        Therefore, we have $P(M) \notin {\bf MI}_1$.

        Suppose that $P(M) \in {\bf MI}_2$. 
        As in (1), we have $2 \le |\Fc_M|\le n=2$, i.e., $|\Fc_M|=2$. Let, say, $\Fc_M=\set{F_1,F_2}$.
        From Lemma~\ref{facets:independent} and Lemma~\ref{lem:parallelclass} (2), we have $\sum_{e\in F_1}x(e) \leq 1$ and $\sum_{e\in F_2}x(e) \leq 1$
        By summing these inequalities, we obtain $\sum_{e\in E}x(e) \leq 2$. However, since $r(M) \geq 2$, this contradicts the fact that the inequality $\sum_{e\in E}x(e) \leq r(E)$ defines a facet of $P(M)$. Therefore, we have $P(M) \notin {\bf MI}_2$.
        
        \item Suppose that $P(M) \in {\bf MI}_3$ then we have $2\le |\Fc_M|\le 3$.
        If $|\Fc_M| = 2$, or $|\Fc_M|=3$ and $r(M)\ge 3$, then, for similar reasons as in (2), $\sum_{e\in E}x(e) \leq r(E)$ does not define a facet of $P(M)$. 
        Assume that $|\Fc_M| = 3$ and $r(M)=2$.
        Let, say, $\Fc_M=\{F_1, F_2, F_3\}$. 
        We will prove that $\{a, b\} \subset E$ is a base of $M$ if and only if $\{a, b\} \not\subset F_i$ for all $i = 1,2,3$. Since $r(F_i) = 1$ for $i = 1,2,3$, any two-element subset of $F_i$ cannot be a base. Therefore, if $\{a, b\}$ is a base, it follows that $\{a, b\} \not\subset F_i$ for $i = 1,2,3$. Now, assume that $a \in F_1, b \in F_2$, and $\{a, b\}$ is not a base of $M$. Since $r(M) = 2$ and $\{a, b\}$ is not a base of $M$, we have $\{a,b\}$ is a circuit of $M$. Moreover, if $\{x,y\} \subset F_1$, then $\{x,y\}$ is also a circuit of $M$. By the properties of circuits, $\{x, b\}$ must also be a circuit of $M$. Since every two-element subset of $F_1 \cup \{b\}$ is a circuit, we have $r(F_1 \cup \{b\}) = 1$, which contradicts the fact that $F_1$ is a flat of $M$.
        Therefore, we have $M=M(\Dc_{s_1,s_2,s_3})$ where $s_1 = |F_1|$, $s_2 = |F_2|$ and $s_3 = |F_3|$.
    \end{enumerate}
\end{proof}

We can see that the graphic matroid $M(\Ac_s)$ coincides with $U_{1,s}$.
Thus, from the above theorem, we immediately obtain the following corollary:
\begin{cor}\label{cor:independence}
    Let $P$ be the independence polytope of a matroid with $\rank P \le 3$.
    Then there exists a graph $G$ with $P\cong P(M(G))$. 
\end{cor}

\begin{rem}\label{rem:false}
    The above corollary is false when $\rank P \ge 4$.
    In fact, letting $M:=U_{2,n}$ with $n\ge 4$, then we can see that $\rank P(M)=n$.
    If there exists a graph $G$ with $P(M(G))\cong P(M)$, then $G$ is a 2-connected graph on 3 vertices since $M$ is a connected matroid with rank 2.
    Thus, we have $G=\Dc_{s_1,s_2,s_3}$ for some $s_1,s_2,s_3\in \ZZ_{>0}$.
    However, it follows from Theorem~\ref{thm: classifyIM} that $\rank P(M(G))=3\neq \rank P(M)$, a contradiction.
\end{rem}

\subsection{Base polytopes with small rank}\label{sec:main2}

Let ${\bf GMB}_n$ denote the sets of unimodular equivalence classes of matroid base polytopes of graphic matroids with rank $n$.

We classify the base polytopes of graphic matroids with small rank by using the operation of adding paths to a cycle. 
The following lemma shows that the classification of the base polytope of a graphic matroid in terms of rank can be reduced to the case of simple graphs:

\begin{lem}\label{lem:adding path}
    Let $G$ be a 2-connected graph and $G'$ be the graph obtained by adding one path $P$ to $G$ with $V(G)\cap V(P)=\{u,v\}$.
    Then we have $\rank B(M(G)) \leq \rank B(M(G'))$.
    In particular, when $\{u,v\}$ is an edge of $G$ and $P$ has length 1, we have $\rank B(M(G')) = \rank B(M(G))$ if $\{u,v\}$ is deletable in $G$, otherwise $\rank B(M(G')) = \rank B(M(G)) + 1$.
\end{lem}
\begin{proof}
    It is clear that a deletable edge in $G$ is also deletable in $G'$.
    Let $H$ be a flacet of $G$.
    If $\{u,v\} \not\subset V(H)$, then $(G/E(H))\cup P$ is 2-connected, which implies that $H$ is a flacet of $G'$.
    On the other hand, if $\{u,v\}\subset V(H)$, then $H':=H \cup P$ is 2-connected and $G'/H'=G/H$, which implies that $H'$ is a flacet of $G'$. Therefore, since the number of flacets of $G'$ is at least that of $G$ and every edge of $P$ is contractible or deletable, we obtain $\rank B(M(G)) \leq \rank B(M(G'))$. 
    
    Assume that $P$ has length 1 and let $H''$ be a flacet of $G'$. 
    It is clear that a deletable edge in $G'$ except for $P$ and $\{u,v\}$ is also deletable in $G$.
    We assume that $\{u,v\}$ is deletable in $G$. If $\{u,v\} \not\subset V(H'')$, then $G/E(H'')$ is 2-connected, which implies that $H''$ is a flacet of $G$. On the other hand, if $\{u,v\}\subset V(H'')$, then $G/(E(H'')\setminus P)$ is 2-connected, which implies that $H'' \setminus P$ is a flacet of $G$. Therefore, we obtain $\rank B(M(G')) \leq \rank B(M(G))$. Hence, we have $\rank B(M(G'))  = \rank B(M(G))$.
    
    We assume that $\{u,v\}$ is not deletable. Since $P$ is a loop in $G'/\{u,v\}$, $\{u,v\}$ is not contractible in $G'$. Similarly, $P$ is also not contractible in $G'$. Moreover, since $\{u,v\}$ is contractible in $G$, $\{u,v\} \cup P$ is a flacet in $G'$. Therefore, we obtain $\rank B(M(G)) + 1 \leq \rank B(M(G'))$. Using a similar argument, we obtain $\rank B(M(G')) - 1 \leq \rank B(M(G))$. Hence, we have $\rank B(M(G')) = \rank B(M(G)) + 1$ .
\end{proof}

\begin{thm}\label{thm:smallsimplebasepolytope}
    Let $G$ be a 2-connected simple graph. Then we have the following:
    \begin{enumerate}[\normalfont(1)]
        \item $B(M(G))\in {\bf GMB}_0$ if and only if $G=\Ac_{s+2}^\ast$ for some $s\in \ZZ_{>0}$.
        \item $B(M(G))\not\in {\bf GMB}_1$.
        \item $B(M(G))\in {\bf GMB}_2$ if and only if $G=\Dc_{s_1+1,s_2+1,1}^\ast$ for some $s_1,s_2\in \ZZ_{>0}$.
        \item $B(M(G))\in {\bf GMB}_3$ if and only if $G$ coincides with one of the following graphs:
        \begin{itemize}
        \item[(i)] $\Bc_{s_1,s_2,s_3,1}^\ast$ for some $s_1,s_2,s_3\in \ZZ_{>0}$;
        \item[(ii)] $\Dc_{s_1+1,s_2+1,s_3+1}^\ast$ for some $s_1,s_2,s_3\in \ZZ_{> 0}$.
        \end{itemize}
    \end{enumerate}
\end{thm}
\begin{proof}
    Since $G$ is 2-connected, $G$ has an ear decomposition.
    If $G=K_2$, then $\dim B(M(G))=0$, so we can assume that $G$ is the graph obtained by adding paths to a cycle graph.
    If $G$ is a cycle graph, then $B(M(G))$ is a simplex, and hence $\rank B(M(G))=0$.
    
    We consider a graph $G_1$ obtained by adding one path to a cycle graph (see Figure~\ref{bicyclic}). 
    %%%%%%%%%%%%%%%%%%%%%%%%%%%%%%%%%%%%%%%%%%%%%%%%%%%%%%%%%%%%%%%%%%%%%%%
    \begin{figure}[H]
        \begin{center}
            \begin{tikzpicture} [scale=0.8]
            \coordinate (v1) at (6,-2) node at (v1) {$G_1$};
            \draw (6,0) circle (1); 
            \draw (5,0)--(7,0);
            \node[draw, circle, fill=white, inner sep=1.5pt] at (5, 0) {};
            \node[draw, circle, fill=white, inner sep=1.5pt] at (7, 0) {};
            \node[] at (4.6,0) {$v_1$};
            \node[] at (7.4,0) {$v_2$};
            \end{tikzpicture}
        \end{center}
        \caption{The graph $G_1$}
        \label{bicyclic}
    \end{figure}
    %%%%%%%%%%%%%%%%%%%%%%%%%%%%%%%%%%%%%%%%%%%%%%%%%%%%%%%%%%%%%%%%%%%%%%%
    Let $v_1$ and $v_2$ be the vertices of $G_1$ with $\deg(v_1)=\deg(v_2)=3$.
    The dual graph of $G_1$ is a multigraph obtained by replacing the edges of $C_3$ with multiple edges, where $C_n$ denotes the cycle graph of length $n$.
    If $\{v_1,v_2\}\in E(G_1)$, then the dual graph $G_1^\ast = \Dc_{s_1+1,s_2+1,1}$ for some $s_1,s_2\in \ZZ_{>0}$. Therefore, we have $\rank B(M(G_1)) = 2$ by Lemma~\ref{lem:adding path}. On the other hand, if $\{v_1,v_2\}\notin E(G_1)$, then the dual graph $G_1^\ast = \Dc_{s_1+1,s_2+1,s_3+1}$ for some $s_1,s_2.s_3 \in \ZZ_{>0}$. Hence, we have $\rank B(M(G_1)) = 3$ by Lemma~\ref{lem:adding path}.    
    
    Next, we consider a graph obtained by adding one path to $G_1$. There are four ways to add a path, resulting in the graph $G_2, G_3, G_4$ and $G_5$, as shown in Figure~\ref{tricyclic} (these graphs are depicted in \cite{voblyi2017enumeration}).
    %%%%%%%%%%%%%%%%%%%%%%%%%%%%%%%%%%%%%%%%%%%%%%%%%%%%%%%%%%%%%%%%%%%%%%%
    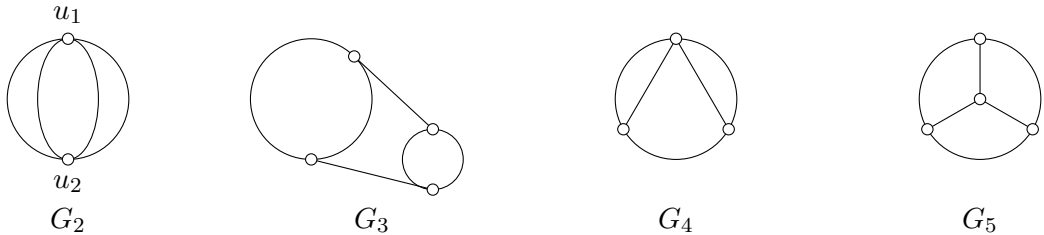
\begin{figure}[H]
        \begin{center}
            \begin{tikzpicture} [scale=0.8]
            \coordinate (v3) at (-10,-6) node at (v3) {$G_2$};
            \draw (-10,-4) circle (1);
            \draw (-10,-4) circle [x radius=0.5,y radius=1];
            \node[draw, circle, fill=white, inner sep=1.5pt] at (-10,-3) {};
            \node[draw, circle, fill=white, inner sep=1.5pt] at (-10,-5) {};
            \node[] at (-10,-2.6) {$u_1$};
            \node[] at (-10,-5.4) {$u_2$};
            
            \coordinate (v1) at (-5,-6) node at (v1) {$G_3$};
            \draw (-6,-4) circle(1);
            \draw (-4,-5) circle(0.5);
            \draw ({-6+sqrt(2)/2},{-4+sqrt(2)/2})--(-4,-4.5);
            \draw (-6,-5)--(-4,-5.5);
            \node[draw, circle, fill=white, inner sep=1.5pt] at ({-6+sqrt(2)/2},{-4+sqrt(2)/2}) {};
            \node[draw, circle, fill=white, inner sep=1.5pt] at (-4, -4.5) {};
            \node[draw, circle, fill=white, inner sep=1.5pt] at (-6, -5) {};
            \node[draw, circle, fill=white, inner sep=1.5pt] at (-4, -5.5) {};

            \coordinate (v2) at (0,-6) node at (v2) {$G_4$};
            \draw (0,-4) circle (1);
            \draw (0,-3)--({-sqrt(3)/2},-4.5);
            \draw (0,-3)--({sqrt(3)/2},-4.5);
            \node[draw, circle, fill=white, inner sep=1.5pt] at (0, -3) {};
            \node[draw, circle, fill=white, inner sep=1.5pt] at ({-sqrt(3)/2},-4.5) {};
            \node[draw, circle, fill=white, inner sep=1.5pt] at ({sqrt(3)/2},-4.5) {};

            \coordinate (v4) at (5,-6) node at (v4) {$G_5$};
            \draw (5,-4) circle (1);
            \draw (5,-4)--(5,-3);
            \draw (5,-4)--({5-sqrt(3)/2},-4.5);
            \draw (5,-4)--({5+sqrt(3)/2},-4.5);
            \node[draw, circle, fill=white, inner sep=1.5pt] at (5,-3) {};
            \node[draw, circle, fill=white, inner sep=1.5pt] at (5,-4) {};
            \node[draw, circle, fill=white, inner sep=1.5pt] at ({5-sqrt(3)/2},-4.5) {};
            \node[draw, circle, fill=white, inner sep=1.5pt] at ({5+sqrt(3)/2},-4.5) {};
            \end{tikzpicture}  
        \end{center}
        \caption{The graphs $G_2$, $G_3$, $G_4$ and $G_5$}
        \label{tricyclic}
    \end{figure}
    %%%%%%%%%%%%%%%%%%%%%%%%%%%%%%%%%%%%%%%%%%%%%%%%%%%%%%%%%%%%%%%%%%%%%%%
    Let $u_1$ and $u_2$ be the vertices of $G_2$ with $\deg(u_1)=\deg(u_2)=4$.
    The dual graph of $G_2$ is $C_4$, in which at least three edges are multiedges. If $\{u_1,u_2\}\in E(G_2)$, then the dual graph $G_2^\ast = \Bc_{s_1,s_2,s_3,1}$ for some $s_1,s_2,s_3\in \ZZ_{>0}$. Therefore, we have $\rank B(M(G_2)) = 3$ by Lemma~\ref{lem:adding path}. On the other hand, if $\{u_1, u_2\}\notin E(G_2)$, then the dual graph $G_2^\ast = \Bc_{s_1,s_2,s_3,s_4,1}$ for some $s_1,s_2,s_3,s_4\in \ZZ_{>0}$. Hence, we have $\rank B(M(G_2)) = 4$ by Lemma~\ref{lem:adding path}.

    The dual graphs of $G_3$ and $G_4$ are obtained by taking a 2-clique sum of $C_3$ and $C_4$ (i.e., identifying an edge of $C_3$ with that of $C_4$), with at least two edges being multiedges. In these graphs, the edges used for the 2-clique sum are deletable, while the other edges are contractible. Furthermore, at least two of the contractible edges are multiedges. Therefore, we have $\rank B(M(G_3)) \geq 4$ and $\rank B(M(G_4)) \geq 4$ by Lemma~\ref{lem:adding path}.

    The dual graph of $G_5$ is a graph obtained by replacing some edges of $K_4$ with multiedges. Since all edges of $K_4$ are both deletable and contractible, it follows that $\rank(B(M(K_4))) \geq 6$, and hence we have $\rank B(M(G_5)) \geq 6$.

    Similarly, adding one path to $G_2$ results in four possible graphs: $G_6, G_7, G_8$ and $G_9$, as shown in Figure~\ref{quadricyclic}.
    %%%%%%%%%%%%%%%%%%%%%%%%%%%%%%%%%%%%%%%%%%%%%%%%%%%%%%%%%%%%%%%%%%%%%%%
    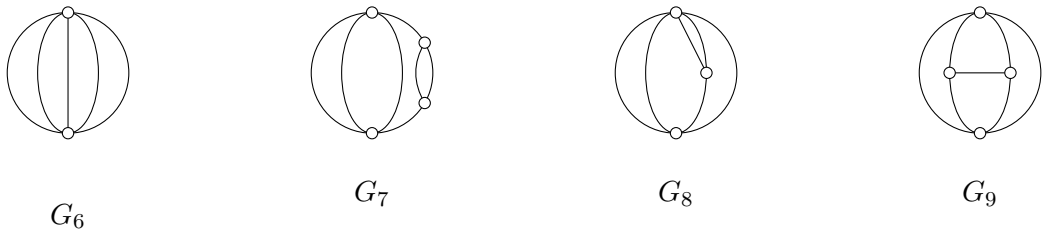
\begin{figure}[H]
        \begin{center}
            \begin{tikzpicture} [scale=0.8]
            \coordinate (v6) at (-10, -6);
            \node[below] at (v6) {$G_6$};
            \draw (-10, -4) circle (1);
            \draw (-10, -4) ellipse (0.5cm and 1cm);
            \draw (-10, -3)--(-10, -5);
            \node[draw, circle, fill=white, inner sep=1.5pt] at (-10,-3) {};
            \node[draw, circle, fill=white, inner sep=1.5pt] at (-10,-5) {};
    
            \coordinate (v7) at (-5,-6) node at (v7) {$G_7$};
            \coordinate (N1) at ({-5+sqrt(3)/2},-7/2);
            \coordinate (N2) at ({-5+sqrt(3)/2},-9/2);
            \draw (-5,-4) circle (1);
            \draw (-5,-4) circle [x radius=0.5,y radius=1];
            \draw (N1) to [out=-120,in=120] (N2);
            \node[draw, circle, fill=white, inner sep=1.5pt] at ({-5+sqrt(3)/2},-7/2) {};
            \node[draw, circle, fill=white, inner sep=1.5pt] at ({-5+sqrt(3)/2},-9/2) {};
            \node[draw, circle, fill=white, inner sep=1.5pt] at (-5,-3) {};
            \node[draw, circle, fill=white, inner sep=1.5pt] at (-5,-5) {};

            \coordinate (v8) at (0,-6) node at (v8) {$G_8$};
            \draw (0,-4) circle (1);
            \draw (0,-4) circle [x radius=0.5,y radius=1];
            \draw (0,-3)--(1/2,-4);
            \node[draw, circle, fill=white, inner sep=1.5pt] at (1/2,-4) {};
            \node[draw, circle, fill=white, inner sep=1.5pt] at (0,-3) {};
            \node[draw, circle, fill=white, inner sep=1.5pt] at (0,-5) {};
            
            \coordinate (v9) at (5,-6) node at (v9) {$G_9$};
            \draw (5,-4) circle (1);
            \draw (9/2,-4)--(11/2,-4);
            \draw (5,-4) circle [x radius=0.5,y radius=1];
            \node[draw, circle, fill=white, inner sep=1.5pt] at (11/2,-4) {};
            \node[draw, circle, fill=white, inner sep=1.5pt] at (9/2,-4) {};
            \node[draw, circle, fill=white, inner sep=1.5pt] at (5,-3) {};
            \node[draw, circle, fill=white, inner sep=1.5pt] at (5,-5) {};
        \end{tikzpicture}  
    \end{center}
    \caption{The graphs $G_6$, $G_7$, $G_8$ and $G_9$}
    \label{quadricyclic}
    \end{figure}
    %%%%%%%%%%%%%%%%%%%%%%%%%%%%%%%%%%%%%%%%%%%%%%%%%%%%%%%%%%%%%%%%%%%%%%%

    The dual graph of $G_6$ is $C_5$, in which at least four edges are multiedges. Therefore, we conclude that $\rank B(M(G_6)) \geq 4$ by Lemma~\ref{lem:adding path}.

    The dual graphs of $G_7$ and $G_8$ are obtained by taking a 2-clique sum of $C_3$ and $C_4$, in which at least four edges are multiedges. Hence, we obtain $\rank B(M(G_7)) \geq 4$ and $\rank B(M(G_8)) \geq 4$ by Lemma~\ref{lem:adding path}.
    
    The dual graph of $G_9$ is obtained by replacing some edges of $G_5$ with multiedges. Therefore, we have $\rank B(M(G_9)) \geq 6$.
    
    From the above discussion and Lemma~\ref{lem:adding path}, the simple graphs satisfying $\rank B(M(G)) \leq 3$ are summarized as follows:
    \begin{itemize}
        \item A simple graph $G$ satisfies $\rank B(M(G)) = 0$ if and only if $G$ is a cycle graph, which shows (1).
        \item There exists no simple graph $G$ such that $\rank B(M(G)) = 1$, which shows (2). 
        \item A simple graph $G$ satisfying $\rank B(M(G)) = 2$ if and only if $G = \Dc_{s_1+1,s_2+1,1}^\ast$ for some $s_1,s_2 \in \ZZ_{> 0}$, which shows (3).
        \item A simple graph $G$ satisfying $\rank B(M(G)) = 3$ if and only if $G = \Dc_{s_1+1,s_2+1,s_3+1}^\ast$ or $G = \Bc_{s_1,s_2,s_3,1}^\ast$ for some $s_1,s_2,s_3 \in \ZZ_{> 0}$, which shows (4).
    \end{itemize}
    
\end{proof}

\begin{thm}\label{thm:smallmultibasepolytope}
    Let $G$ be a 2-connected multigraph. Then we have the following:
    \begin{enumerate}[\normalfont(1)]
        \item $B(M(G))\in {\bf GMB}_0$ if and only if $B(M(G))$ is unimodularly equivalent to the matroid base polytope $B(M(\Ac_{s+1}))$ for some $s\in \ZZ_{>0}$.
        \item $B(M(G))\in {\bf GMB}_1$ if and only if $B(M(G))$ is unimodularly equivalent to the matroid base polytope $B(M(\Bc_{s,p}))$ for some $s,p\in \ZZ_{>0}$.
        \item $B(M(G))\in {\bf GMB}_2$ if and only if $B(M(G))$ is unimodularly equivalent to the matroid base polytope $B(M(\Bc_{s_1,s_2,p}))$ for some $s_1,s_2\in \ZZ_{>0}$ and $p\in \ZZ_{\ge 0}$.
        \item $B(M(G))\in {\bf GMB}_3$ if and only if $B(M(G))$ is unimodularly equivalent to one of the following matroid base polytopes:
        \begin{itemize}
        \item[(i)] $B(M(\Bc_{s_1,s_2,s_3,p}))$ for some $s_1,s_2,s_3\in \ZZ_{>0}$ and $p\in \ZZ_{\ge 0}$;
        \item[(ii)] $B(M(\Cc_{s,t,p,q}))$ for some $s,q\in \ZZ_{>0}$ and $t,p\in \ZZ_{\ge 0}$;
        \item[(iii)] $B(M(\Dc_{s_1+1,s_2+1,s_3+1}))$ for some $s_1,s_2,s_3\in \ZZ_{> 0}$.
        \end{itemize}
    \end{enumerate}
\end{thm}
\begin{proof}
    \begin{enumerate}[\normalfont(1)]
        \item Since no edge of a cycle graph is deletable, we have $\rank B(M(\Bc_{s_1,\ldots,s_n,p})) = n$ for $n\in \ZZ_{>0}$, $s_1,\ldots,s_n\in \ZZ_{>0}$ and $p\in \ZZ_{\ge 0}$ from Lemma~\ref{lem:adding path}.
        Therefore, from Theorem~\ref{thm:smallsimplebasepolytope}, $\rank B(M(G)) \in {\bf GMB}_0$ if and only if $B(M(G))$ is unimodularly equivalent to $B(M(\Ac_{s+1}))$ for some $s\in \ZZ_{>0}$.
        \item Since no simple graph $G$ satisfies $\rank B(M(G)) = 1$ from Theorem~\ref{thm:smallsimplebasepolytope}, the only graphs that satisfy $\rank B(M(G)) = 1$ are those in which exactly one edge of a cycle graph is a multiedge.
        From the above discussion and Theorem~\ref{thm:smallsimplebasepolytope}, we have $\rank B(M(G)) \in {\bf GMB}_1$ if and only if $B(M(G))$ is unimodularly equivalent to $B(M(\Bc_{s,p}))$ for some $s,p\in \ZZ_{>0}$.
        \item In the graph $G_1$, the edge $\{v_1,v_2\}$ is deletable, while all other edges are not deletable. The dual of the graph in which $\{v_1,v_2\}$ is a multiedge coincides with the graph $\Bc_{s_1,s_2,p}$ for some $s_1,s_2\in \ZZ_{>0}$ and $p\in \ZZ_{\ge 0}$.
        From the above discussion and Theorem~\ref{thm:smallsimplebasepolytope}, we have $\rank B(M(G)) \in {\bf GMB}_2$ if and only if $B(M(G))$ is unimodularly equivalent to $B(M(\Bc_{s_1,s_2,p}))$ for some $s_1,s_2\in \ZZ_{>0}$ and $p\in \ZZ_{\ge 0}$.
        \item In the graph $G_1$, if an edge other than the one with $\{v_1,v_2\}$ becomes a multiedge, then the rank of the base polytope of the graphic matroid of the resulting graph is 3, where the graph is $\Cc_{s,t,p,q}$ for some $s,q \in \ZZ_{>0}$ and $t,p \in \ZZ_{\ge 0}$.
        On the other hand, if $\{u,v\} \not\in E(G_1)$, then this graph has no deletable edges. Therefore, by Theorem~\ref{thm:smallsimplebasepolytope} and Lemma~\ref{lem:adding path}, if this graph has a multiedge, then $\rank B(M(G_1)) \geq 4$.
        Furthermore, in the graph $G_2$, the edge with $\{u_1,u_2\}$ is deletable, while all other edges are not deletable. The dual of the graph in which $\{u_1,u_2\}$ is a multiedge coincides with the graph $\Bc_{s_1,s_2,s_3,p}$ for some $s_1,s_2,s_3 \in \ZZ_{> 0}$ and $p \in \ZZ_{\ge 0}$. From the above discussion and Theorem~\ref{thm:smallsimplebasepolytope}, $\rank B(M(G)) \in {\bf GMB}_3$ if and only if $B(M(G))$ is unimodularly equivalent to $B(M(\Bc_{s_1,s_2,s_3,p}))$ or $B(M(\Dc_{s_1+1,s_2+1,s_3+1}))$ for some $s_1,s_2,s_3 \in \ZZ_{> 0}$ and $p \in \ZZ_{\ge 0}$, or to $B(M(\Cc_{s,t,p,q}))$ for some $s,q \in \ZZ_{>0}$ and $t,p \in \ZZ_{\ge 0}$.
    \end{enumerate}
\end{proof}

Let ${\bf MB}_n$ denote the sets of unimodular equivalence classes of matroid base polytopes with rank $n$.
Clearly, we have ${\bf GMB}_n\subset {\bf MB}_n$.
It is natural to consider whether the results we obtained for ${\bf GMB}_n$ can be generalized to the case of ${\bf MB}_n$.
Actually, the concept of ear decomposition also exists for matroids (see \cite{Coullard1996Independence}).
However, unlike in the case of graphs, we have no method for classifying matroids with a small number of ears, so another means would be needed.

On the other hand, it is also natural to ask how different ${\bf GMB}_n$ and ${\bf MB}_n$ are.
By a similar argument as in Remark~\ref{rem:false}, we can see that $B(U_{2,n})\in {\bf MB}_n\setminus {\bf GMB}_n$ for $n\ge 4$.

\begin{q}
    Does the relationship ${\bf MB}_n={\bf GMB}_n$ hold if $n\le 3$?
\end{q}

\section{The relationships among our polytopes}\label{sec:relation}

In this section, we discuss the relationships among ${\bf MI}_n, {\bf GMB}_n, {\bf Order}_n$, ${\bf Stab}_n$ and ${\bf Edge}_n$ 
in the cases $n=0,1,2,3$ by using the results in the previous section. 

\subsection{Order polytopes, stable set polytopes and edge polytopes}
In this subsection, we recall order polytopes, stable set polytopes and edge polytopes. 

\bigskip

\noindent {\large {\bf Order polytopes}}.
Let $\Pi$ be a finite partially ordered set (poset, for short) equipped with a partial order $\preceq$. 
For a subset $I \subset \Pi$, we say that $I$ is a \textit{poset ideal} of $\Pi$ if $p \in I$ and $q \preceq p$ imply $q \in I$.
For a subset $A \subset \Pi$, we call $A$ an \textit{antichain} of $\Pi$ if $p \not\preceq q$ and $q \not\preceq p$ for any $p,q \in A$ with $p \neq q$. 
Note that $\emptyset$ is regarded as a poset ideal and an antichain.

For a poset $\Pi:=\{p_1,\ldots,p_d\}$, let \begin{align*}
\Oc_\Pi:=\{x \in \RR^\Pi : \; x(p_i) \geq x(p_j) \text{ if } p_i \preceq p_j \text{ in }\Pi, \;\; 
0 \leq x(p_i) \leq 1 \text{ for }i=1,\ldots,d \}.
\end{align*}
A convex polytope $\Oc_\Pi$ is called the \textit{order polytope} of $\Pi$. 
It is known that the vertices of $\Oc_\Pi$ one-to-one correspond to the poset ideals of $\Pi$, that is, $\set{x\in \RR^\Pi : x \text{ is a vertex of }\Oc_\Pi}=\set{\chi_I\in \RR^\Pi : I \text{ is a poset ideal of }\Pi}$ (\cite{stanley1986twoposet}). 
Moreover, the order polytope $\Oc_\Pi$ has IDP (see \cite{ohsugi2001compressed}).
%In fact, a $(a_1,\ldots,a_d)$ is a vertex of $\Oc_\Pi$ if and only if $\{ p_i \in \Pi : a_i =1 \}$ is a poset ideal. 

\smallskip

We also recall another polytope arising from $\Pi$, which is defined as follows: 
\begin{align*}
\Cc_\Pi:=\{x \in \RR^\Pi : \;&x(p_i) \geq 0 \text{ for }i=1,\ldots,d, \\
&\sum_{j=1}^kx(p_{i_j}) \leq 1 \text{ for }p_{i_1} \preceq \cdots \preceq p_{i_k} \text{ in }\Pi\}.\end{align*} 
A convex polytope $\Cc_\Pi$ is called the \textit{chain polytope} of $\Pi$. 
Similarly to order polytopes, it is known that the vertices of $\Cc_\Pi$ one-to-one correspond to the antichains of $\Pi$, that is, $\set{x\in \RR^\Pi : x \text{ is a vertex of }\Cc_\Pi}=\set{\chi_A\in \RR^\Pi : A \text{ is an antichain of }\Pi}$ (\cite{stanley1986twoposet}). 
%In fact, a $(0,1)$-vector $(a_1,\ldots,a_{d-1})$ is a vertex of $\Oc_\Pi$ if and only if $\{ p_i \in \Pi : a_i =1 \}$ is a poset ideal. 

\smallskip

In general, the order polytope and the chain polytope of $\Pi$ are not unimodularly equivalent, 
but the following is known: 

\begin{thm}[{\cite[Theorem 2.1]{Hibi2016unimodular}}]\label{X}
Let $\Pi$ be a poset. Then $\Oc_\Pi$ and $\Cc_\Pi$ are unimodularly equivalent if and only if $\Pi$ does not contain the X-shape subposet. 
\end{thm}
Here, the \textit{X-shape poset} is the poset $\{z_1,z_2,z_3,z_4,z_5\}$ equipped with the partial orders $z_1 \preceq z_3 \preceq z_4$ and $z_2 \preceq z_3 \preceq z_5$. 

\bigskip

\noindent {\large {\bf Stable set polytopes}}.
Let $G$ be a finite simple graph on the vertex set $V(G)$ with the edge set $E(G)$.%, where we let $[d]=\{1,\ldots,d\}$ for $d \in \ZZ_{>0}$. 
We say that $T \subset V(G)$ is a \textit{stable set} or an \textit{independent set} (resp. a \textit{clique}) 
if $\{v,w\} \not\in E(G)$ (resp. $\{v,w\} \in E(G)$) for any distinct vertices $v,w \in T$. 
Note that the empty set and each singleton are regarded as independent sets.
 
We define a lattice polytope associated with a graph $G$ as follows: 
\begin{align*}
\Stab_G:=\conv(\{\chi_T : T \text{ is a stable set of }G\})\subset \RR^{V(G)}. 
\end{align*}
We call $\Stab_G$ the \textit{stable set polytope} of $G$. 

In this paper, we focus on the stable set polytopes of perfect graphs (we say that a graph $G$ is \textit{perfect} if for any induced subgraph $H$ of $G$, the chromatic number of $H$ is equal to the maximum size of cliques of $H$). 
% The reason why we focus on perfect graphs is derived from the following: 
% \begin{itemize}
If $G$ is perfect, then $\Stab_G$ has IDP (see \cite{ohsugi2001compressed}) and the facets of $\Stab_G$ are completely characterized (\cite[Theorem 3.1]{chvatal1975stable}). 
More concretely, the facet defining inequalities of $\Stab_G$ can be written as follows: 
\begin{equation*}\label{facets:stab}
\begin{split}
&x(i) \ge 0 \;\; \text{ for each } i \in V(G); \\ 
&1- \sum_{i\in Q}x(i)\ge 0 \;\;\text{ for each maximal clique }Q. 
\end{split}
\end{equation*}
From this description, we can see that the stable set polytope of a perfect graph $G$ has rank 0 if and only if $G=K_n$ for some $n\in \ZZ_{>0}$.
Moreover, for two graphs $G$ and $G'$, we can see that $\Stab_G\times\Stab_{G'}$ is unimodularly equivalent to $\Stab_{G\sqcup G'}$ by observing the stable sets of $G\sqcup G'$, where $G\sqcup G'$ denotes the graph on the vertex set $V(G)\sqcup V(G')$ with the edge set $E(G)\sqcup E(G')$.

Given a poset $\Pi$, we define the \textit{comparability graph} of $\Pi$, denoted by $G(\Pi)$, 
as a graph on the vertex set $V(G(\Pi))$ with the edge set 
$$E(G(\Pi))=\{\{p_i,p_j\} : \text{$p_i$ and $p_j$ are comparable in $\Pi$}\}.$$ 
It is known that $G(\Pi)$ is perfect for any $\Pi$ (see, e.g., \cite[Section 5.5]{diestel2017graph}). 
Moreover, we see that $\Cc_\Pi=\Stab_{G(\Pi)}$. 

\bigskip

\noindent {\large {\bf Edge polytopes}}.
%For a positive integer $d$, 
Consider a graph $G$ on the vertex set $V(G)$ with the edge set $E(G)$. 
We define a lattice polytope associated to $G$ as follows: 
\begin{align*}
P_G=\conv(\{\chi_e : e \in E(G)\})\subset \RR^{V(G)}. 
\end{align*}
We call $P_G$ the \textit{edge polytope} of $G$. 

We have that $\dim P_G=|V(G)|-b(G)-1$, where $b(G)$ is the number of bipartite connected components of $G$ (see \cite[Proposition 10.4.1]{stanley1986twoposet}). 
It is known that $P_G$ has IDP if and only if $G$ satisfies the \textit{odd cycle condition}, i.e., for each pair of odd cycles $C$ and $C'$ with no common vertex, there is an edge $\{v,v'\}$ with $v \in V(C)$ and $v' \in V(C')$ (see \cite[Corollary 10.3.11]{stanley1986twoposet}).

\bigskip

Let ${\bf Order}_n$, ${\bf Stab}_n$ and ${\bf Edge}_n$ denote the sets of unimodular equivalence classes of order polytopes, stable set polytopes of perfect graphs and edge polytopes of graphs satisfying the odd cycle condition with rank $n$, respectively.

\begin{thm}[{\cite{matsushita2022three}}]\label{thm:relationOSE}
    The following relationships hold:
\begin{itemize}
\setlength{\parskip}{0pt} 
\setlength{\itemsep}{3pt}
\item ${\bf Order}_n={\bf Stab}_n={\bf Edge}_n$ if $n\le 1$;
\item ${\bf Stab}_2 \cup {\bf Edge}_2={\bf Order}_2$ and no inclusion between ${\bf Stab}_2$ and ${\bf Edge}_2$;
\item there is no inclusion among ${\bf Order}_n$, ${\bf Stab}_n$ and ${\bf Edge}_n$ if $n\ge 3$.
\end{itemize}
\end{thm}

\bigskip

We define a poset as follows:
For $s,q \in \ZZ_{>0}$ and $t,p \in \ZZ_{\geq 0}$,  let 
$$W_{s,t,p,q}=\{\alpha_1,\ldots,\alpha_s,\beta_1,\ldots,\beta_p,\gamma_1,\ldots,\gamma_t,\delta_1,\ldots,\delta_q,\mu_1,\mu_2,\mu_3\}$$ be the poset equipped with the partial orders
\begin{itemize}
\item $\mu_1 \prec \alpha_1 \prec \cdots \prec \alpha_s$,
\item $\mu_1 \prec \beta_1 \prec \cdots \prec \beta_p \prec \mu_2$, 
\item $\mu_3 \prec \gamma_1 \prec \cdots \prec \gamma_t \prec \mu_2$ and
\item $\mu_3 \prec \delta_1 \prec \cdots \prec \delta_q$.
\end{itemize}
Figure~\ref{posetW} shows the Hasse diagram of $W_{s,t,p,q}$. 
%\end{itemize}

%%%%%%%%%%%%%%%%%%%%%%
%\input{figure_poset}
\begin{figure}[ht]
    \centering
    {\scalebox{0.7}{
\begin{tikzpicture}[line width=0.05cm]

\coordinate (N11) at (0,4); 
\coordinate (N12) at (0.4,3); 
\coordinate (N13) at (1.2,1); 
\coordinate (N15) at (1.6,0); 
\coordinate (N21) at (2,1); 
\coordinate (N22) at (2.8,3); 
\coordinate (N23) at (3.2,4); 
\coordinate (N25) at (3.6,3); 
\coordinate (N26) at (4.4,1);
\coordinate (N31) at (4.8,0); 
\coordinate (N32) at (5.2,1); 
\coordinate (N33) at (6,3); 
\coordinate (N35) at (6.4,4); 

%edge

\draw (N11)--(N12);
\draw (N12)--(0.6,2.5);
\draw[dotted] (0.68,2.3)--(0.92,1.7);
\draw (1.0,1.5)--(N13);
\draw (N13)--(N15);
\draw (N15)--(N21);
\draw (N21)--(2.2,1.5);
\draw[dotted] (2.28,1.7)--(2.52,2.3);
\draw (2.6,2.5)--(N22);
\draw (N22)--(N23);
\draw (N23)--(N25);
\draw (N25)--(3.8,2.5);
\draw[dotted] (3.88,2.3)--(4.12,1.7);
\draw (4.2,1.5)--(N26);
\draw (N26)--(N31);
\draw (N31)--(N32);
\draw (N32)--(5.4,1.5);
\draw[dotted] (5.48,1.7)--(5.72,2.3);
\draw (5.8,2.5)--(N33);
\draw (N33)--(N35);

%node
\draw [line width=0.05cm, fill=white] (N11) circle [radius=0.15]
node[below left] {\Large $\alpha_s$}; 
\draw [line width=0.05cm, fill=white] (N12) circle [radius=0.15]
node[below left] {\Large $\alpha_{s-1}$}; 
\draw [line width=0.05cm, fill=white] (N13) circle [radius=0.15]
node[below left] {\Large $\alpha_1$}; 
\draw [line width=0.05cm, fill=white] (N15) circle [radius=0.15]
node[] at (1.6,-0.5) {\Large $\mu_1$}; 
\draw [line width=0.05cm, fill=white] (N21) circle [radius=0.15]
node[below right] {\Large $\beta_1$};
\draw [line width=0.05cm, fill=white] (N22) circle [radius=0.15]
node[above left] {\Large $\beta_p$}; 
\draw [line width=0.05cm, fill=white] (N23) circle [radius=0.15]
node[] at (3.2,4.5) {\Large $\mu_2$}; 
\draw [line width=0.05cm, fill=white] (N25) circle [radius=0.15]
node[above right] {\Large $\gamma_t$}; 
\draw [line width=0.05cm, fill=white] (N26) circle [radius=0.15]
node[below left] {\Large $\gamma_1$}; 
\draw [line width=0.05cm, fill=white] (N31) circle [radius=0.15]
node[] at (4.8,-0.5) {\Large $\mu_3$};
\draw [line width=0.05cm, fill=white] (N32) circle [radius=0.15]
node[below right] {\Large $\delta_1$}; 
\draw [line width=0.05cm, fill=white] (N33) circle [radius=0.15]
node[below right] {\Large $\delta_{q-1}$}; 
\draw [line width=0.05cm, fill=white] (N35) circle [radius=0.15]
node[below right] {\Large $\delta_q$};
\end{tikzpicture} }}
\caption{The poset $W_{s,t,p,q}$}
\label{posetW}
\end{figure}
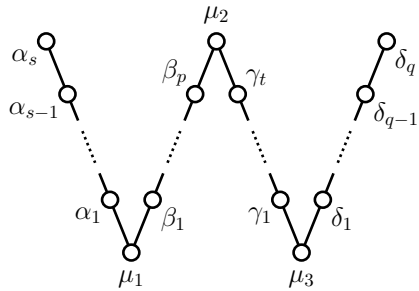

%%%%%%%%%%%%%%%%%%%%%%

In addition, for $n\in \ZZ_{>0}$, $s_1,\ldots,s_n\in \ZZ_{>0}$ and $p\in \ZZ_{\ge 0}$, we define the graph $G_{s_1,\ldots,s_n,p}$ as the graph obtained by identifying one vertex of $K_{s_i+1}$ for each $i\in [1,n]$ with $n$ distinct vertices of $K_{p+n}$, respectively. See Figure~\ref{graphG}.
Let us consider the vertex set $V(G_{s_1,\ldots,s_n,p})=\rbra{\bigcup_{i\in [1,n]}V(K_{s_i+1})} \cup V(K_{p+n})$, such that 
\begin{align*}
    V(K_{s_i+1})=\set{v_{i,1},\ldots,v_{i,s_i},m_i}, \quad V(K_{p+n})=\set{m_1,\ldots,m_n,u_1,\ldots,u_p}.
\end{align*}
Note that $G_{s_1,\ldots,s_n,p}$ is perfect (see \cite[Propositions~5.5.1 and 5.5.2]{diestel2017graph}).

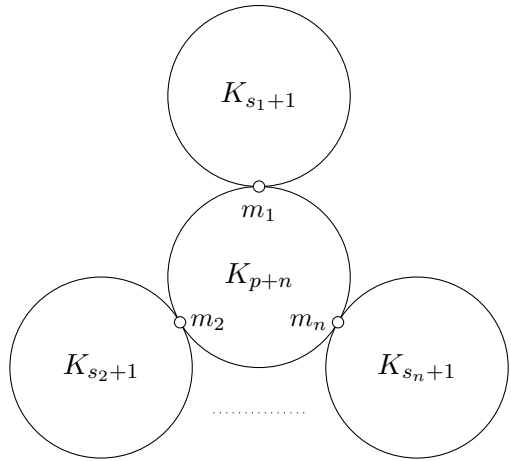
\begin{figure}[h]
\begin{tikzpicture}[scale=1.2]
    % 外側の円
    \draw (0, 2) circle (1) node {$K_{s_1+1}$}; % 上
    \draw (-1.732, -1) circle (1) node {$K_{s_2+1}$}; % 左下
    \draw (1.732, -1) circle (1) node {$K_{s_n+1}$}; % 右下

    % 中央の円
    \draw (0, 0) circle (1) node {$K_{p+n}$};

    % 接続点
    \node[draw, circle, fill=white, inner sep=1.5pt] at (0, 1) {}; % 上と中央
    \node[] at (0,0.7) {\small $m_1$};
    \node[draw, circle, fill=white, inner sep=1.5pt] at (-0.866, -0.5) {}; % 左下と中央
    \node[right] at (-0.866, -0.5) {\small $m_2$};
    \node[draw, circle, fill=white, inner sep=1.5pt] at (0.866, -0.5) {}; % 右下と中央
    \node[left] at (0.866, -0.5) {\small $m_n$};
    \draw[dotted] (-0.5,-1.5)--(0.5,-1.5);
\end{tikzpicture}
\caption{The graph $G_{s_1,\ldots,s_n,p}$}
\label{graphG}
\end{figure}

Let $M$ be a matroid with the ground set $E$.
For a nonempty subset $A\subset E$ and $a\in A$, we define the maps $f_{A,a}:\RR^E\to \RR^E$ and $g_A:\RR^E\to \RR^E$ as follows:
\begin{itemize}
    \item For $x\in \RR^E$ and $e\in E$, let 
    $$(f_{A,a}(x))(e):=\begin{cases}
        1-\sum_{i\in A}x(i) &\text{if }e=a \\
        x(e) &\text{otherwise}.
    \end{cases}$$
    \item For $x\in \RR^E$ and $e\in E$, let
        $$(g_A(x))(e):=\begin{cases}
        1-x(e) &\text{if }e\in A \\
        x(e) &\text{otherwise}.
        \end{cases}$$
\end{itemize}
Moreover, for $e\in E$, let $\pi_e : \RR^E\to \RR^{E\setminus\set{e}}$ denote the projection which ignores the $e$-th coordinate, i.e., for $x\in \RR^E$, $\pi_e(x)$ is the restriction $x|_{E\setminus\set{e}}$ of $x$ to $E\setminus{\set{e}}$.

\begin{lem}\label{lem:equiBS}
For $n\in \ZZ_{>0}$, $s_1,\ldots,s_n\in \ZZ_{>0}$ and $p\in \ZZ_{\ge 0}$, the base polytope $B(M(\Bc_{s_1,\ldots,s_n,p}))$ is unimodularly equivalent to the stable set polytope $\Stab_{G_{s_1,\ldots,s_n,p}}$.  
\end{lem}
\begin{proof}
    For $i\in [1,n]$ and $k_i\in [1,s_i+1]$, let $T_{k_1,\ldots,k_n}:=\set{e_{1,k_1},\ldots,e_{n,k_n},\epsilon_0,\ldots,\epsilon_p}$.
    In addition, let $E_i:=\set{e_{i,1},\ldots,e_{i,s_i+1}}$ for each $i\in [1,n]$ and let $E:=\set{\epsilon_0,\ldots,\epsilon_p}$.
%Notice that $\Tc_1\cup \Tc_2\cup \Tc_3$ is the set of spanning forests of $\Cc_{s,t,p}$.
Then any spanning forest of $\Bc_{s_1,\ldots,s_n,p}$ can be written as $T_{k_1,\ldots,k_n}\setminus\set{e}$ for some $e\in T_{k_1,\ldots,k_n}$.
We set 
$$f:=\pi_{\epsilon_p}\circ f_{E_n,e_{n,s_n+1}}\circ \cdots \circ f_{E_1,e_{1,s_1+1}}\circ g_E.$$ Then we can see that 
 \begin{equation}\label{equ:B}
f(\chi_{T_{k_1,\ldots,k_i}\setminus\set{e}})=\begin{cases} \chi_{\set{e_{1,k_1},\ldots,\overset{i}{\check{e_{i,s_i+1}}},\ldots,e_{n,k_n}}\setminus\set{e_{1,s_1+1},\ldots,\widehat{e_{i,s_i+1}},\ldots,e_{n,s_n+1}}} & \text{ if } e=e_{i,k_i} \text{ for some $i$} \\ \chi_{\set{e_{1,k_1},\ldots,e_{n,k_n},\epsilon_k}\setminus\set{e_{1,s_1+1},\ldots,e_{n,s_n+1},\epsilon_p}} & \text{ if } e=\epsilon_k \text{ for some $k$},\end{cases}
\end{equation}
where $\widehat{\;}$ indicates an element to be omitted.
On the other hand, we consider the bijection $\phi : E(\Bc_{s_1,\ldots,s_n,p})\setminus\set{\epsilon_p} \to V(G_{s_1,\ldots,s_n,p})$ defined by the following correspondence:
\begin{align*}
    &e_{i,k_i}\mapsto v_{i,k_i} \quad (i\in [1,n] \text{ and }k_i\in [1,s_i]), &e_{i,s_i+1}\mapsto m_i \quad (i\in [1,n] \text{ and }k_i\in [1,s_i]), &\; &\; \\ &\epsilon_i\mapsto u_{i+1} \quad (i\in [0,p-1]).
\end{align*}
This bijection induces the one-to-one correspondence between the characteristic vectors described in (\ref{equ:B}) and the characteristic vectors of the stable sets of $G_{s_1,\ldots,s_n,p}$ as follows:
\begin{equation*}
\begin{split}
   &\chi_{\set{e_{1,k_1},\ldots,\overset{i}{\check{e_{i,s_i+1}}},\ldots,e_{n,k_n}}\setminus\set{e_{1,s_1+1},\ldots,\widehat{e_{i,s_i+1}},\ldots,e_{n,s_n+1}}} \mapsto 
   \chi_{\set{v_{1,k_1},\ldots,\overset{i}{\check{m_i}},\ldots,v_{n,k_n}}\setminus\set{m_1,\ldots,\widehat{m_i},\ldots,m_n}}, \\
   &\chi_{\set{e_{1,k_1},\ldots,e_{n,k_n},\epsilon_k}\setminus\set{e_{1,s_1+1},\ldots,e_{n,s_n+1},\epsilon_p}} \mapsto \begin{cases} \chi_{\set{v_{1,k_1},\ldots,v_{n,k_n},u_{k+1}}\setminus\set{m_1,\ldots,m_n}} &\text{ if } k\neq p \\
   \chi_{\set{v_{1,k_1},\ldots,v_{n,k_n}}\setminus\set{m_1,\ldots,m_n}} &\text{ if } k= p,\end{cases}
   \end{split}
\end{equation*}
where we let $v_{i,s_i+1}=m_i$ for each $i\in [1,n]$.
This shows that $B(M(\Bc_{s_1,\ldots,s_n,p}))$ is unimodularly equivalent to $\Stab_{G_{s_1,\ldots,s_n,p}}$.

\end{proof}

\subsection{The cases $n=0$ or $1$}

First, we discuss the relationship in the cases $n=0$ or 1.

\begin{prop}\label{prop:n=01}
    We have the following:
    \begin{enumerate}[\normalfont(1)]
        \item ${\bf MI}_0={\bf GMB}_0={\bf Stab}_0$.
        \item ${\bf MI}_1\subsetneq {\bf GMB}_1\subsetneq {\bf Stab}_1$.
    \end{enumerate}
\end{prop}
\begin{proof}
    (1)  By Theorems~\ref{thm: classifyIM} and \ref{thm:smallmultibasepolytope}, any independence (resp. base) polytope with rank 0 is unimodularly equivalent to $P(M(\Ac_{s}))$ (resp. $B(M(\Ac_{s+1}))$) for some $s\in \ZZ_{>0}$.
    Similarly, any stable set polytope with rank 0 is unimodularly equivalent to $\Stab_{K_t}$ for some $t\in \ZZ_{>0}$.
    We can see that 
    $P(M(\Ac_{s}))=\pi_{e_1}(B(M(\Ac_{s+1})))=\Cc_{I_s}$ for any $s\in \ZZ_{>0}$,
    which implies that ${\bf MI}_0={\bf GMB}_0={\bf Stab}_0$.
    (2) It follows from Theorems~\ref{thm: classifyIM} and \ref{thm:smallmultibasepolytope} that any independence (resp. base) polytope with rank 1 is unimodularly equivalent to $P(M(\Ac_{s_1}\sqcup\Ac_{s_2}))$ (resp. $B(M(\Ac_{s_1+1}\sqcup\Ac_{s_2+1}))$ or $B(M(\Bc_{s_1,p}))$) for some $s_1,s_2\in \ZZ_{>0}$ (resp. $s_1,s_2,p\in \ZZ_{>0}$).
    As we observed in (1), $P(M(\Ac_{s_1}\sqcup\Ac_{s_2}))$, $B(M(\Ac_{s_1+1}\sqcup\Ac_{s_2+1}))$ and $\Stab_{K_{s_1}\sqcup K_{s_2}}$ are unimodularly equivalent.
    Moreover, by Lemma~\ref{lem:equiBS}, $B(M(\Bc_{s_1,p}))$ and $\Stab_{G_{s_1,p}}$ are unimodularly equivalent for any $s_1,p\in \ZZ_{>0}$.
    These imply that ${\bf MI}_1\subset {\bf GMB}_1\subset {\bf Stab}_1$.
    We can see that $B(M(\Bc_{1,1}))\notin {\bf MI}_1$. Indeed, if there exist positive integers $s_1$ and $s_2$ such that $P(M(\Ac_{s_1}\sqcup\Ac_{s_2}))$ is unimodularly equivalent to $B(M(\Bc_{1,1}))$, then we have $s_1+s_2=3$ since $\dim P(M(\Ac_{s_1}\sqcup\Ac_{s_2}))=\dim B(M(\Bc_{1,1}))=3$.
    Thus, we may assume that $s_1=1$ and $s_2=2$, and {\tt MAGMA} confirms that $P(M(\Ac_{s_1}\sqcup\Ac_{s_2}))$ is not unimodularly equivalent to $B(M(\Bc_{1,1}))$.

    Let $G$ be the graph on the vertex set $\set{1,2,3,4}$ with the edge set $$\set{\set{1,2},\set{1,3},\set{2,3},\set{2,4},\set{3,4}}.$$
    Similar to the discussion above, we can check by {\tt MAGMA} that $\Stab_G\in {\bf Stab}_1$ and there is no base polytope which is unimodularly equivalent to $\Stab_G$, that is, $\Stab_G\notin {\bf GMB}_1$.
    Therefore, we have ${\bf MI}_1\subsetneq {\bf GMB}_1\subsetneq {\bf Stab}_1$
\end{proof}

\subsection{The case $n=2$}

Next, we discuss the relationship in the case $n=2$.

\begin{prop}\label{prop:n=2}
    We have the following:
    \begin{enumerate}[\normalfont(1)]
        \item ${\bf MI}_2\subsetneq {\bf GMB}_2\subsetneq {\bf Stab}_2$.
        \item For any $\Sc \in \set{{\bf MI}_2,{\bf GMB}_2}$, there is no inclusion between $\Sc$ and ${\bf Edge}_2$.
    \end{enumerate}
\end{prop}
\begin{proof}
    (1) First, we show that ${\bf MI}_2\subset {\bf GMB}_2$. Let $M$ be a matroid with $\rank P(M)=2$. If we can write $M$ as the direct sum of two matroids $M_1$ and $M_2$ with $\rank P(M_i)\le 1$ for $i=1,2$, then there exist matroids $M'_1$ and $M'_2$ such that $B(M'_i)$ is unimodularly equivalent to $P(M_i)$ for each $i=1,2$ from Proposition~\ref{prop:n=01}, and hence $P(M)\in {\bf GMB}_2$.
    Thus, we may assume that $M$ is connected. However, in this case, there is no such matroid $M$ by Theorem~\ref{thm: classifyIM}, which shows that ${\bf MI}_2\subset {\bf GMB}_2$.

    Next, we prove that ${\bf GMB}_2\subset {\bf Stab}_2$.
    Let $M'$ be a matroid with $\rank B(M')=2$.
    For the same reason as above, we may assume that $M'$ is connected.
    Then we can see that $B(M')$ is unimodularly equivalent to $B(M(\Bc_{s_1,s_2,p}))$ from Theorem~\ref{thm:smallmultibasepolytope} (3) and that $B(M(\Bc_{s_1,s_2,p}))$ is unimodularly equivalent to $\Stab_{G_{s_1,s_2,p}}$, which implies that ${\bf GMB}_2\subset {\bf Stab}_2$.

    Finally, we show that the above inclusions are proper.
    Let $M=M(\Bc_{1,1,0})$ and let $G$ be the graph on the vertex set $\set{1,2,3,4,5}$ with the edge set $$\set{\set{1,2},\set{1,3},\set{2,3},\set{2,4},\set{3,4}}.$$
    Then we have $B(M)\notin {\bf MI}_2$ and $\Stab_G\notin {\bf GMB}_2$.
    Their proofs can be performed in the similar way to Proposition~\ref{prop:n=01}. 

    \medskip

    (2) Since ${\bf MI}_2\subset {\bf GMB}_2$, it is enough to show that there exist graphs $G$ and $H$ with $\rank P_G=\rank P(M(H))=\rank B(M(H))=2$ such that $P_G \not\in {\bf GMB}_2$ and $P(M(H)) \not\in {\bf Edge}_2$, respectively.
    Let $G$ be the graph on the vertex set $V(G)=\set{1,\ldots,7}$ with the edge set
    $$E(G)=\set{\set{1,2},\set{1,7},\set{2,6},\set{3,4},\set{4,7},\set{5,6},\set{5,7},\set{6,7}},$$ which is the same as the graph in \cite[Figure~17]{matsushita2022three}, and let $H=\bigcup_{i=1}^3\Ac_1$.
    It is shown in \cite[Section~5.2]{matsushita2022three} that $P_G\in {\bf Edge}_2\setminus {\bf Stab}_2$.
    Since ${\bf GMB}_2\subset {\bf Stab}_2$, we have $P_G\notin {\bf GMB}_2$.
    Moreover, we can check by using {\tt MAGMA} that $P(M)\in {\bf MI}_2\setminus{\bf Edge}_2$. 
    
\end{proof}

\subsection{The case $n=3$}

Finally, we discuss the relationship in the case $n=3$.
Before that, we provide the following lemmas:

\begin{lem}\label{lem:CcongO}
    For $t,q\in \ZZ_{>0}$ and $s,p\in \ZZ_{\ge 0}$, the base polytope $B(M(\Cc_{s,t,p,q}))$ is unimodularly equivalent to the order polytope $\Oc_{W_{s,t,p,q}}$
\end{lem}
\begin{proof}
Let $E_1:=\set{e_{1,i} ; i\in [1,s+1]}$, $E_2:=\set{e_{2,i} : i\in [1,t+1]}$, $\Ec_1:=\set{\epsilon_{1,i} : i\in [0,p]}$ and $\Ec_2:=\set{\epsilon_{2,i} : i\in [0,q]}$.
For $i\in [1,s+1]$ and $j\in [1,t+1]$, we set $T_{i,j}:=\{e_{1,i},e_{2,j}\}\cup \Ec_1\cup\Ec_2$.
Then a spanning forest of $\Cc_{s,t,p,q}$ is one of the following forms:
\begin{itemize}
    \item[(T1)] $T_{i,j}\setminus \set{e_{1,i},e_{2,j}}$;
    \item[(T2)] $T_{i,j}\setminus \set{e_{1,i},\epsilon_{2,l}}$ for some $l\in [0,q]$;
    \item[(T3)] $T_{i,j}\setminus \set{e_{2,j},\epsilon_{1,k}}$ for some $k\in [0,p]$;
    \item[(T4)] $T_{i,j}\setminus \set{e_{2,j},\epsilon_{2,l}}$ for some $l\in [0,q]$;
    \item[(T5)] $T_{i,j}\setminus \set{\epsilon_{1,k},\epsilon_{2,l}}$ for some $k\in [0,p]$ and $l\in [0,q]$.
\end{itemize}
We define the map $h : \RR^{E(\Cc_{s,t,p,q})} \to \RR^{E(\Cc_{s,t,p,q})}$ as follows:
$$(h(x))(e)=\begin{cases}
    x(e_{2,t+1})-\sum_{i\in E_2\setminus\set{e_{2,t+1}}}x(i) &\text{ if }e=e_{2,t+1} \\
    x(e) &\text{ otherwise. }
\end{cases}$$
In addition, let $f:=\pi_{\epsilon_{1,0}}\circ h \circ f_{\set{e_{2,t+1},\epsilon_{2,0}},e_{2,t+1}} \circ f_{\Ec_2,\epsilon_{2,0}}\circ f_{E_1,e_{1,s+1}}\circ g_{\Ec_2}\circ g_{E_1}$.
Then we can see that
\begin{equation}\label{equ:C}
\begin{split}
f(\chi_{T_{i,j}\setminus\set{e_{1,i},e_{2,j}}})&=\chi_{\set{e_{1,s+1},\epsilon_{2,0}}}, \\
f(\chi_{T_{i,j}\setminus \set{e_{1,i},\epsilon_{2,l}}})&=\chi_{\set{e_{1,s+1},e_{2,j},\epsilon_{2,l}}\setminus \set{e_{2,t+1},\epsilon_{2,0}}}, \\
f(\chi_{T_{i,j}\setminus \set{e_{2,j},\epsilon_{1,k}}})&=\chi_{\set{e_{1,i},\epsilon_{1,k},\epsilon_{2,0}}\setminus \set{e_{1,s+1},\epsilon_{1,0}}}, \\
f(\chi_{T_{i,j}\setminus \set{e_{2,j},\epsilon_{2,l}}})&=\chi_{\set{e_{1,i},e_{2,t+1},\epsilon_{2,l}}\setminus\set{e_{1,s+1},\epsilon_{2,0}}},\\ 
f(\chi_{T_{i,j}\setminus \set{\epsilon_{1,k},\epsilon_{2,l}}})&=\chi_{\set{e_{1,i},e_{2,j},\epsilon_{1,k},\epsilon_{2,l}}\setminus \set{e_{1,s+1},e_{2,t+1},\epsilon_{1,0},\epsilon_{2,0}}}.
\end{split}
\end{equation}
On the other hand, we consider the bijection $\phi : E(\Cc_{s,t,p,q})\setminus\set{\epsilon_{1,0}} \to W_{s,t,p,q})$ defined by the following correspondence:
\begin{align*}
    e_{1,i}\mapsto \alpha_{i} \quad &(i\in [1,s]), \quad &e_{1,s+1}\mapsto \mu_1, \quad &\; \quad &e_{2,j} \mapsto \gamma_j \quad &(j\in [1,t]), \\
    e_{2,t+1}\mapsto \mu_2, &\;  &\epsilon_{1,k}\mapsto \beta_{k} \quad &(k\in [1,p]), \quad &\epsilon_{2,l} \mapsto \delta_{l} \quad &(l\in [1,q]), \\
    \epsilon_{2,0}\mapsto \mu_3. &\; &\; &\; &\; &\;
\end{align*}
This bijection induces the one-to-one correspondence between the characteristic vectors described in (\ref{equ:C}) and the characteristic vectors of antichains of $W_{s,t,p,q}$.
Therefore, $B(M(\Cc_{s,t,p,q}))$ is unimodularly equivalent to $\Cc_{W_{s,t,p,q}}$, which is also unimodularly equivalent to $\Oc_{W_{s,t,p,q}}$ from Theorem~\ref{X}.
\end{proof}

\begin{lem}\label{lem:DcongE}
     For $s_1,s_2,s_3\in \ZZ_{>0}$, the base polytope $B(M(\Dc_{s_1,s_2,s_3}))$ is unimodularly equivalent to the edge polytope $P_{K_{s_1,s_2,s_3}}$.
\end{lem}
\begin{proof}
For each $i=1,2,3$, let $E_i:=\set{e_{i,k} : k\in [1,s_i]}$.
Then any spanning forest of $\Dc_{s_1,s_2,s_3}$ can be written as $\set{e_{i,k},e_{j,l}}$ for some $1\le i<j\le 3$, $k\in [1,s_i]$ and $l\in [1,s_j]$, which one-to-one corresponds to the edge $\set{k,l}$ of $K_{s_1,s_2,s_3}$ where $k\in V_i$ and $l\in V_j$.
This shows that $B(M(\Dc_{s_1,s_2,s_3}))$ coincides with $P_{K_{s_1,s_2,s_3}}$.
\end{proof}

\begin{prop}\label{prop:n=3}
     We have the following:
     \begin{enumerate}[\normalfont(1)]
         \item For any $\Sc \in \set{{\bf Order}_3,{\bf Stab}_3,{\bf Edge}_3,{\bf GMB}_3}$, there is no inclusion between $\Sc$ and ${\bf MI}_3$.
         In particular, we have ${\bf MI}_3 \not\subset {\bf Order}_3\cup{\bf Stab}_3\cup{\bf Edge}_3\cup{\bf GMB}_3$.
         \item For any $\Tc \in \set{{\bf Order}_3,{\bf Stab}_3,{\bf Edge}_3,{\bf MI}_3}$, there is no inclusion between $\Tc$ and ${\bf GMB}_3$.
         On the other hand, we have ${\bf GMB}_3\subsetneq {\bf Stab}_3\cup{\bf Edge}_3$.
     \end{enumerate}
\end{prop}
\begin{proof}
    (1) We can see that $P(M(\Dc_{0,0,0}))\notin {\bf Order}_3\cup{\bf Stab}_3\cup{\bf Edge}_3\cup{\bf GMB}_3$.
    In fact, $P(M(\Dc_{0,0,0}))$ is a 3-dimensional lattice polytope with 7 vertices, but {\tt MAGMA} confirms that there is no 3-dimensional polytope with 7 vertices which belongs to ${\bf Order}_3\cup{\bf Stab}_3\cup{\bf Edge}_3\cup{\bf GMB}_3$.
    
    Let $\Pi:=\set{p_1,p_2,p_3,p_4}$ be the poset equipped with the partial orders $p_1\prec p_3$, $p_1\prec p_4$, $p_2\prec p_3$ and $p_2\prec p_4$. 
    By using {\tt MAGMA}, we obtain $\Oc_\Pi\in {\bf Order}_3$ and $\Oc_\Pi\in {\bf Stab}_3$ from Theorem~\ref{X}.
    If there exists a matroid $M$ such that $P(M)$ is unimodularly equivalent to $\Oc_\Pi$, then $|E(M)|=\dim \Oc_\Pi=4$.
    We can check that the independence polytopes of such matroids are not unimodularly equivalent to $\Oc_\Pi$. 

    Let $G=K_{2,2,2}$. From \cite[Theorem~1.3]{matsushita2022conic} and Lemma~\ref{lem:DcongE}, we have $P_{G}\in {\bf Edge}_3$ and $P_{G}\in {\bf GMB}_3$.
    If there exists a matroid $M$ such that $P(M)$ is unimodularly equivalent to $P_{G}$, then $|E(M)|=\dim P_{G}=5$.
    {\tt MAGMA} guarantees that the independence polytopes of such matroids are not unimodularly equivalent to $P_{G}$.
    Thus, there is no inclusion between $\Sc$ and ${\bf MI}_3$ for any $\Sc \in \set{{\bf Order}_3,{\bf Stab}_3,{\bf Edge}_3,{\bf GMB}_3}$.

    \medskip

    (2) Let $\TT:=\set{{\bf Order}_3,{\bf Stab}_3,{\bf Edge}_3,{\bf MI}_3}$.
    We already see in (1) that there is no inclusion between ${\bf GMB}_3$ and ${\bf MI}_3$, and that $P_{K_{2,2,2}}\in {\bf GMB}_3$.
    From \cite[Section~5.3]{matsushita2022three}, we have $P_{K_{2,2,2}}\notin {\bf Order}_3\cup {\bf Stab}_3$, so ${\bf GMB}_3 \not\subset{\bf Order}_3\cup {\bf Stab}_3$.
    Moreover, we can check that $B(M(\bigsqcup_{i=1}^4\Ac_2))\notin {\bf Edge}_3$ by focusing on the dimension of $B(M(\bigsqcup_{i=1}^4\Ac_2))$ and using {\tt MAGMA}.
    Thus, we have ${\bf GMB}_3\not\subset \Tc$ for any $\Tc\in \TT$.

    Conversely, let $\Pi$ be the poset defined in (1) above and let $G$ be the graph on the vertex set $V(G)=\set{1,2,3,4,5}$ with the edge set
    $$E(G)=\set{\set{1,2},\set{2,3},\set{3,4},\set{4,5},\set{1,5},\set{1,3},\set{1,4}},$$
    then we have $\Oc_\Pi\in \rbra{{\bf Order}_3\cap {\bf Stab}_3}\setminus {\bf GMB}_3$ and $P_G\in {\bf Edge}_3 \setminus {\bf GMB}_3$. Their proofs can be performed in the similar way to the above discussions.
    Therefore, we get $\Tc \not\subset {\bf GMB}_3$ for any $\Tc\in \TT$.

    Finally, we show that ${\bf GMB}_3\subsetneq {\bf Stab}_3\cup{\bf Edge}_3$.
    Let $M$ be a graphic matroid with $\rank B(M)=3$.
    If $M$ has at least 2 connected components, then $B(M)\in {\bf Stab}_3$ from Propositions~\ref{prop:n=01} and \ref{prop:n=2}.
    If $M$ is connected, then $B(M)$ is unimodularly equivalent to a base polytope of the form (i), (ii) or (iii) in Theorem~\ref{thm:smallmultibasepolytope} (4).
    By Lemmas~\ref{lem:equiBS}, \ref{lem:CcongO} and \ref{lem:DcongE}, we obtain $B(M)\in {\bf Stab}_3\cup{\bf Edge}_3$, and hence ${\bf GMB}_3\subsetneq {\bf Stab}_3\cup{\bf Edge}_3$.
\end{proof}

\bibliographystyle{plain} 
\bibliography{ref}

\begin{thebibliography}{10}

\bibitem{aliniaeifard2018stable}
Farid Aliniaeifard, Carolina Benedetti, Nantel Bergeron, Shu~Xiao Li, and Franco Saliola.
\newblock Stable set polytopes and their 1-skeleta.
\newblock {\em arXiv preprint arXiv:1804.00360}, 2018.

\bibitem{benedetti2023lattice}
Carolina Benedetti, Kolja Knauer, and Jer{\'o}nimo Valencia-Porras.
\newblock On lattice path matroid polytopes: alcoved triangulations and snake decompositions.
\newblock {\em arXiv preprint arXiv:2303.10458}, 2023.

\bibitem{Blasiak2008TheToric}
Jonah Blasiak.
\newblock The toric ideal of a graphic matroid is generated by quadrics.
\newblock {\em Combinatorica}, 28(3):283--297, 2008.

\bibitem{bruns2009polytopes}
Winfried Bruns and Joseph Gubeladze.
\newblock {\em Polytopes, rings, and K-theory}.
\newblock Springer Science \& Business Media, 2009.

\bibitem{chvatal1975stable}
V~Chv\'{a}tal.
\newblock On certain polytopes associated with graphs.
\newblock {\em J. Combin. Theory, Ser. B}, 18:138--154, 1975.

\bibitem{Coullard1996Independence}
Collette~R. Coullard and Lisa Hellerstein.
\newblock Independence and port oracles for matroids, with an application to computational learning theory.
\newblock {\em Combinatorica}, 16(2):189--208, 1996.

\bibitem{diestel2017graph}
Reinhard Diestel.
\newblock {\em Graph theory, Fifth edition}, volume 173.
\newblock Springer, 2017.

\bibitem{feichtner2004matroid}
Eva~Maria Feichtner and Bernd Sturmfels.
\newblock Matroid polytopes, nested sets and {B}ergman fans.
\newblock {\em Port. Math. (N.S.)}, 62(4):437--468, 2005.

\bibitem{hall2023nearly}
Thomas Hall, Max K\"olbl, Koji Matsushita, and Sora Miyashita.
\newblock Nearly {G}orenstein polytopes.
\newblock {\em Electron. J. Combin.}, 30(4):Paper No. 4.42, 21, 2023.

\bibitem{Herzog2002Discrete}
J\"urgen Herzog and Takayuki Hibi.
\newblock Discrete polymatroids.
\newblock {\em J. Algebraic Combin.}, 16(3):239--268, 2002.

\bibitem{hibi2021gorenstein}
Takayuki Hibi, Micha{\l} Laso{\'n}, Kazunori Matsuda, Mateusz Micha{\l}ek, and Martin Vodi{\v{c}}ka.
\newblock {G}orenstein graphic matroids.
\newblock {\em Israel Journal of Mathematics}, 243(1):1--26, 2021.

\bibitem{Hibi2016unimodular}
Takayuki Hibi and Nan Li.
\newblock Unimodular equivalence of order and chain polytopes.
\newblock {\em Math. Scand.}, 118(1):5--12, 2016.

\bibitem{matsushita2022conic}
Akihiro Higashitani and Koji Matsushita.
\newblock Conic divisorial ideals and non-commutative crepant resolutions of edge rings of complete multipartite graphs.
\newblock {\em J. Algebra}, 594:685--711, 2022.

\bibitem{matsushita2022three}
Akihiro Higashitani and Koji Matsushita.
\newblock Three families of toric rings arising from posets or graphs with small class groups.
\newblock {\em J. Pure Appl. Algebra}, 226(10):Paper No. 107079, 24, 2022.

\bibitem{kim2010flag}
Sangwook Kim.
\newblock Flag enumerations of matroid base polytopes.
\newblock {\em J. Combin. Theory Ser. A}, 117(7):928--942, 2010.

\bibitem{kolbl2020gorenstein}
Max K{\"o}lbl.
\newblock {G}orenstein graphic matroids from multigraphs.
\newblock {\em Annals of Combinatorics}, 24(2):395--403, 2020.

\bibitem{Lason2014OnThe}
Micha\l{} Laso\'n and Mateusz Micha\l~ek.
\newblock On the toric ideal of a matroid.
\newblock {\em Adv. Math.}, 259:1--12, 2014.

\bibitem{Lason2023GorMat}
Micha\l{} Laso\'n and Mateusz Micha\l~ek.
\newblock Gorenstein matroids.
\newblock {\em Int. Math. Res. Not. IMRN}, (18):15687--15728, 2023.

\bibitem{matsushita20240/1}
Koji Matsushita.
\newblock Toric rings of {$(0,1)$}-polytopes with small rank.
\newblock {\em Illinois J. Math.}, 68(2):281--307, 2024.

\bibitem{ohsugi2001compressed}
Hidefumi Ohsugi and Takayuki Hibi.
\newblock Convex polytopes all of whose reverse lexicographic initial ideals are squarefree.
\newblock {\em Proc. Amer. Math. Soc.}, 129(9):2541--2546, 2001.

\bibitem{oxley2011matroid}
James Oxley.
\newblock {\em Matroid Theory}.
\newblock Oxford University Press, 02 2011.

\bibitem{schrijver2003combinatorial}
Alexander Schrijver et~al.
\newblock {\em Combinatorial optimization: polyhedra and efficiency}, volume~24.
\newblock Springer, 2003.

\bibitem{stanley1986twoposet}
Richard~P Stanley.
\newblock Two poset polytopes.
\newblock {\em Discrete Comput. Geom.}, 1:9--23, 1986.

\bibitem{villarreal2001monomial}
Rafael~H. Villarreal.
\newblock {\em Monomial algebras}.
\newblock Monographs and Research Notes in Mathematics. CRC Press, Boca Raton, FL, second edition, 2015.

\bibitem{voblyi2017enumeration}
V.~A. Voblyi and A.~K. Meleshko.
\newblock Enumeration of labeled outerplanar bicyclic and tricyclic graphs.
\newblock {\em Diskretn. Anal. Issled. Oper.}, 24(2):18--31, 2017.

\bibitem{white1977basis}
Neil~L White.
\newblock The basis monomial ring of a matroid.
\newblock {\em Advances in Mathematics}, 24(2):292--297, 1977.

\bibitem{white1980AUnique}
Neil~L. White.
\newblock A unique exchange property for bases.
\newblock {\em Linear Algebra Appl.}, 31:81--91, 1980.

\end{thebibliography}

\end{document}